\documentclass[a4paper,reqno]{amsart}
\usepackage{amssymb,latexsym,amsthm,paralist,dsfont}
\usepackage{amsmath}
\usepackage{amsfonts}
\usepackage{mathrsfs}
\usepackage[OT2,T1]{fontenc}
\usepackage{tikz-cd}
\usepackage{graphicx}
\usepackage{color}
\usepackage{float}
\usetikzlibrary{babel}
\usepackage{varioref}
\usepackage[pdfpagelabels, pdftex]{hyperref}
\hypersetup{
  pdftitle={Poitou-Tate duality for arithmetic schemes},
  pdfauthor={Thomas H. Geisser and Alexander Schmidt},
  pdfsubject={},
  pdfkeywords={Poitou-Tate duality, Arithmetic scheme, Shafarevich-Tate group},
  colorlinks=true,    
  linkcolor=red,     
  citecolor=blue,     
  filecolor=blue,      
  urlcolor=blue,       
  breaklinks=true,
  bookmarksopen=true,
  bookmarksnumbered=true,
  pdfpagemode=UseOutlines,
  plainpages=false}
\usepackage[capitalise]{cleveref}

\definecolor{darklimegreen}{RGB}{31,142,8}

\newcommand{\bC}{{\mathds C}}

\newcommand{\bF}{{\mathds F}}
\newcommand{\bG}{{\mathds G}}
\newcommand{\bH}{{\mathds H}}

\newcommand{\bR}{{\mathds R}}

\newcommand{\cB}{{\mathscr B}}

\newcommand{\cD}{{\mathscr D}}

\newcommand{\cF}{{\mathscr F}}
\newcommand{\cG}{{\mathscr G}}

\newcommand{\cS}{{\mathscr S}}
\newcommand{\cT}{{\mathscr T}}
\newcommand{\cU}{{\mathscr U}}

\newcommand{\cW}{{\mathscr W}}
\newcommand{\cX}{{\mathscr X}}

\DeclareSymbolFont{cyrletters}{OT2}{wncyr}{m}{n}
\DeclareMathSymbol{\Sha}{\mathalpha}{cyrletters}{"58}

\newtheoremstyle{alexthm}
  {\bigskipamount}
  {\bigskipamount}
  {\sl }
  {}
  {\bf}
  {.}
  {.5em}
  {}
\theoremstyle{alexthm}
\newtheorem{theorem}{Theorem}[section]
\newtheorem{corollary}[theorem]{Corollary}
\newtheorem{proposition}[theorem]{Proposition}
\newtheorem{lemma}[theorem]{Lemma}

\newtheorem{thmABC}{Theorem}

\newtheoremstyle{alexdef}
  {\bigskipamount}
  {\bigskipamount}
  {\rm }
  {}
  {\bf}
  {.}
  {.5em}
  {}
\theoremstyle{alexdef}

\newtheorem{definition}[theorem]{Definition}
\newtheorem{remark}[theorem]{Remark}
\newtheorem{notconv}[theorem]{Convention}

\numberwithin{equation}{section}

 \newcommand{\im}{\mathrm{im}}
 
 \DeclareMathOperator{\Spec}{Spec}
 \newcommand{\Gal}{\operatorname{Gal}}

  \newcommand{\Ind}{\operatorname{Ind}}

 \newcommand{\Hom}{\operatorname{Hom}}
 \newcommand{\IHom}{\operatorname{\mathcal{H}\!\mathit{om}}}
 \newcommand{\IExt}{\mathcal{E}\!\mathit{xt}}
 
  \newcommand{\coker}{\operatorname{coker}}
  \newcommand{\Tot}{\operatorname{Tot}}

\newcommand{\et}{\mathit{et}}
\newcommand{\nr}{\mathit{nr}}

 \newcommand{\Q}{{\mathds  Q}}
 \newcommand{\Z}{{\mathds  Z}}

 \renewcommand{\O}{\mathcal{O}}

\newcommand{\T}{\mathcal{T}}

\newcommand{\Br}{\operatorname{Br}}
\newcommand{\lang}{\longrightarrow}

\newcommand{\Sh}{\mathit{Sh}}

\newcommand{\Ext}{\operatorname{Ext}}
\newcommand{\F}{\mathds{F}}

\newcommand{\ds}{\displaystyle}

\newcommand{\G}{\mathds{G}}
\DeclareMathOperator{\inv}{\mathit{inv}}

\DeclareMathOperator{\tr}{tr}
\newcommand{\R}{\mathrm R}

\DeclareMathOperator{\cone}{cone}

\DeclareMathOperator{\codim}{codim}
\DeclareMathOperator{\ch}{\mathit{char}}

\newcommand{\liso}{\mathrel{\hbox{$\longrightarrow$} \kern-2.4ex\lower-1ex\hbox{$\scriptstyle\sim$}\kern1.6ex}}
\newcommand{\kiso}{\mathrel{\hbox{$\rightarrow$} \kern-1.9ex\lower-1ex\hbox{$\scriptstyle\sim$}\kern.5ex}}
\newcommand{\resprodsign}{\setbox0 = \hbox{$\ds\prod$}
    \hbox{\vtop{\copy0\kern -1.6pt \hrule\vspace*{1pt}}}}   
\newcommand{\rresprodsign}{\setbox0 = \hbox{$\prod$}          
    \hbox{\vtop{\copy0\kern -.4pt\hrule}}}         

\newcommand{\resprod}{\mathop{\mathchoice%
{\resprodsign}
{\rresprodsign}
{\scriptstyle \rresprodsign}
{\scriptscriptstyle \rresprodsign}
}\limits}

\DeclareRobustCommand{\SkipTocEntry}[5]{}

\begin{document}

\title{\bf Poitou-Tate duality for arithmetic schemes}
\author{Thomas H. Geisser}
\address{Rikkyo University, Department of Mathematics, 3-34-1 Nishi-Ikebukuro, Toshima-ku,
Tokyo Japan 171-8501}
\email{geisser@rikkyo.ac.jp}
\author{Alexander Schmidt}
\address{Mathematisches  Institut, Universit\"{a}t Heidelberg, Im Neuenheimer Feld 205, 69120 Heidelberg, Deutschland}
\email{schmidt@mathi.uni-heidelberg.de}
\date{}
\maketitle

\noindent
The classical Poitou-Tate theorem considers the cohomology of Galois groups with restricted ramification of global fields. It states a perfect duality between Sha\-fa\-re\-vich-Tate groups, and a $9$-term exact sequence relating global and local cohomology groups, cf.\ \cite[(8.6.7),\,(8.6.10)]{NSW}.
We prove the following generalization to higher dimensional schemes (see \cref{sec1} for the notation).

\medskip
Let $S$ be a nonempty set of places of a global field $k$, and assume that $S$ contains the set $S_\infty$ of archimedean places if\/ $k$ is a number field. Let $\O_S$ be the ring of $S$-integers in $k$ and $\cS=\Spec \O_S$. Let $\cX\to \cS$ be a regular, flat, and separated scheme of finite type of relative dimension~$r$, $m\ge 1$ an integer invertible on $\cS$ and $\cF$ a locally constant, constructible sheaf of $\Z/m\Z$-modules on $\cX$.
We consider the \emph{Shafarevich-Tate groups} defined by
\begin{eqnarray*}
  \Sha^i(\cX,\cS,\cF) &=& \ker\big( H^i_\et(\cX, \cF) \to \textstyle\prod_{v\in S} \hat H^i_\et(\cX \otimes_{\O_S} k_v,\cF)\big) \\
  \Sha^i_c(\cX,\cS,\cF) &=& \ker\big( H^i_c(\cX, \cF) \to \textstyle\prod_{v\in S} \hat H^i_c(\cX \otimes_{\O_S} k_v,\cF)\big).
\end{eqnarray*}
\begin{thmABC}[Poitou-Tate duality]\label{PTdual} Under the assumptions above, the Shafarevich-Tate groups are finite and there are perfect pairings for $i=0,\ldots,2r+2$:
\begin{equation}
 \Sha^i(\cX,\cS,\cF)\ \times \ \Sha^{2r+3-i}_c(\cX,\cS,\cF^\vee(r+1)) \lang \Q/\Z.
\end{equation}
\end{thmABC}
Recall that a homomorphism $\varphi: A \to B$ between topological groups is  \emph{strict} if it is continuous and the isomorphism $A/\ker(\varphi)\kiso \im (\varphi)\subset B$ is a homeomorphism. It is called \emph{proper} if preimages of compact sets are compact.
\begin{thmABC}[Poitou-Tate exact sequence] \label{PTseq}  For $\cX$, $\cS$ and $\cF$ as above, we have an exact $6r+9$-term sequence of abelian topological groups and strict homomorphisms
\begin{equation}\label{longexactPT}
\begin{tikzcd}[row sep=tiny, column sep=tiny]
0\rar&H^0_\et(\cX, \cF)\rar{} & P^0(\cX,\cF)\rar & H^{2r+2}_c(\cX,\cF^\vee(r+1))^\vee\rar&\null\\
&&\cdots &\\
\cdots\rar&H^i_\et(\cX, \cF)\rar{\lambda_i} & P^i(\cX,\cF) \rar& H^{2r+2-i}_c(\cX,\cF^\vee(r+1))^\vee\rar&\null\\
&&\cdots &\\
\cdots\rar&H^{2r+2}_\et(\cX, \cF)\rar{} & P^{2r+2}(\cX,\cF)\rar & H^{0}_c(\cX,\cF^\vee(r+1))^\vee\rar&0.
\end{tikzcd}
\end{equation}
Here
\begin{equation}
P^i(\cX,\cF):= \resprod_{v\in S} \hat H^i_\et(\cX \otimes_{\O_S} k_v,\cF)
\end{equation}
is the restricted product with respect to the subgroups $H^i_{\nr}(\cX \otimes_{\O_S} k_v,\cF)$ (see  \cref{defunramifiedcoh}). The localization map $\lambda_i$ is proper and has finite kernel for all $i$, and for $i\ge 2r+3$,
\begin{equation}
\lambda_i: H^i_\et (\cX, \cF) \liso P^i(\cX,\cF)=\prod_{v\in S_\infty} \hat H^i_\et(\cX \otimes_{\O_S} k_v,\cF)
\end{equation}
is an isomorphism. The groups in the left column of (\ref{longexactPT}) are discrete, those in the middle column locally compact, and those in the right column compact.
\end{thmABC}

If $\cX$ is a smooth variety over $k$ (i.e., $S$=all places),  \cref{PTseq} was proven by S.~Saito~\cite{Sa}.
His proof combines classical Poitou-Tate duality with the fact that
$
Rf^!(\Z/m\Z)\cong\mu_m^{\otimes d}[2d]
$
for any smooth, geometrically connected morphism $f:X\to Y$ of schemes with $m$ invertible on $Y$ \cite[XVIII Th.\,(3.2.5)]{sga4}. With mild effort, Saito's argument can be extended to the case that $\cX\to \cS$ is \emph{smooth}. The essential new achievement of this paper is that the assumption on $\cX$ can be weakened from smooth to regular. We overcome the technical difficulty with $Rf^!$ by making a detour to algebraic cycle complexes, which have good base change properties by \cite{Ge}. Furthermore, we prove a compactly supported version of \cref{PTseq} and a version that applies to singular schemes and non-invertible coefficients as well.

In our context it is technically more convenient to work with henselizations rather than completions. Therefore we will first prove our results in their henselian versions and will pass to completions in the final \cref{vollstsec}.

The authors thank the anonymous referee for his thorough reading and constructive comments.

\section{Notation and conventions} \label{sec1} In this paper we use the following notation:

\medskip
\begin{compactitem}
\item $k$ a global field, $p=\ch (k)$
\item $k_v$ (resp.\ $k_{(v)}$) the completion (resp.\ henselization) of $k$ at a place $v$
\item $\O_v$ (resp.\ $\O_{(v)}$) the ring of integers of $k_v$ (resp.\ of $k_{(v)}$) (if $v$ is nonarchimedean)
\item $S_\infty$ the set of archimedean places if $k$ is a number field ($S_\infty=\varnothing$ otherwise)
\item $S_\bR\subset S_\infty$ the set of real places
\item $S$ a set of places of $k$ containing $S_\infty$
\item $\O_S$ the ring of $S$-integers in $k$ (if $S\ne \varnothing$)
\item $\cS=\Spec \O_S$, if $p>0$  and $S=\varnothing$ then $\cS$ is the unique regular, complete curve with function field $k$
\item $\cX \to \cS$ a separated scheme of finite type over $\cS$.
\item $X_v=\cX\otimes_{\O_S} k_v$ for a place $v$ of $k$, analogous: $X_{(v)}=\cX\otimes_{\O_S} k_{(v)}$
\item $\cX_v= \cX \otimes_{\O_S} \O_v$ for a nonarchimedean place $v$ of $k$; $\cX_{(v)}= \cX \otimes_{\O_S} \O_{(v)}$
\item $A^\vee$ the Pontryagin dual of a locally compact abelian group
\item  ${\resprod}_{i\in I} (A_i,B_i)$ the restricted product of a family $(A_i)_{i\in I}$ of abelian Hausdorff topological groups with respect to open subgroups $B_i \subset A_i$ given for almost all $i\in I$ (the $B_i$ are omitted when clear from the context)
\item $m$ a natural number invertible on $\cS$
\item $\cF$ an \'{e}tale sheaf of $\Z/m\Z$-modules on $\cX_\et$
\item $\cF^\vee(r)$ the $r$-th Tate twist of the dual sheaf $\IHom(\cF,\Z/m\Z)$
\item $X_\et$ the small \'{e}tale site of a scheme $X$
\item $H^*_\et(X,-)$ the \'{e}tale cohomology $X$
\item $H^*_c(X,-)$ the \'{e}tale cohomology with compact support of a scheme $X$, separated and of finite type over some base scheme $B$
\end{compactitem}

\medskip\noindent
If $v$ is nonarchimedean and the fibre of $\cX$ over $v$ is empty, then $X_v\to \cX_v$ is a scheme isomorphism. However,   $H^*_c(X_v,-)$ and $H^*_c(\cX_v,-)$ differ because the base scheme is different.

The modification of a cohomology theory $H^*$ with respect to the real places is denoted by $\hat H^*$, see \cref{sec2}.

For a scheme $X$, we denote by $\Sh(X_\et)$ the category of sheaves of abelian groups on the small \'{e}tale site on $X$.
We call a complex $\cF^\bullet$ of sheaves in  $\Sh(X_\et)$  bounded if almost all of its cohomology sheaves are zero. We call $\cF^\bullet$ locally constant, torsion, constructible, etc., if all cohomology sheaves have this property.

\section{Modified cohomology} \label{sec2}

In this section we extend the definition of the modified (or \lq\lq compactly supported\rq\rq) \'{e}tale cohomology of an \'{e}tale sheaf on a number ring defined by Th.~Zink \cite{Zi} to bounded complexes of sheaves. Our construction involves a cone. As a cone is only well defined for an actual morphism of complexes (and not for a morphism in the derived category), we work with Godement resolutions to obtain a functorial model for hypercohomology.

\medskip
Let $G$ be a finite group and let $A$ be a $G$-module. We let $A^{tr}$ be the trivial $G$-module with underlying abelian group $A$ and
$
A_1:= \coker (\iota: A \to \Ind^{\{1\}}_G A^{tr}),
$
where $\Ind^{\{1\}}_G A^{tr}$ is the induced module $\bigoplus_{\sigma\in G} A^{tr}_\sigma$ (on which $G$ acts by interchanging the summands) and $\iota(a)= (\sigma a)_\sigma$.
We set $A_0=A$ and recursively $A_{n+1}:=(A_n)_1$ for $n\ge 0$ to obtain a functorial resolution
\begin{equation}\label{Godement1}
A \lang C^\bullet(A), \quad C^n(A)= \Ind^{\{1\}}_G A_n^{tr}
\end{equation}
of $A$ by induced, hence cohomologically trivial $G$-modules. Hence the hypercohomology $\bH(G,A)$ is naturally isomorphic to $C^\bullet(A)^G$.

If $G=\Gal(\bC|\bR)$, we can interpret $G$-modules as sheaves on $(\Spec \bR)_\et$ and (\ref{Godement1}) is nothing  but the Godement resolution
\begin{equation}\label{Godement2}
\cF \lang C^\bullet(\cF),
\end{equation}
 which is defined by $C^n(\cF)=i_*i^*\cF_n$ with $i: \Spec \bC \to \Spec \bR$ the canonical morphism and $
\cF_0=\cF$, $\cF_{n+1}:= \coker  (\cF_n\to i_*i^*\cF_n)$, cf.\ \cite[III, 1.20]{Mi}.

\medskip
Returning to the case of a general finite group $G$, let
$A_{-1}= \ker(\varepsilon: \Ind^{\{1\}}_G A^{tr} \to A)$ with $\varepsilon (a_\sigma)= \sum \sigma^{-1}a_\sigma$. We put $A_0=A$ and recursively $A_{n-1}:=(A_n)_{-1}$ for $n\le 0$  to obtain a left resolution
\begin{equation}\label{Godement3}
 D^\bullet(A)\to A, \quad D^n(A)= \Ind^{\{1\}}_G A_n^{tr}.
\end{equation}
We splice these resolutions together to obtain a \emph{functorial} complete acyclic resolution $\hat C^\bullet(A)$ of $A$
\begin{equation}\label{defhat}
\hat C^n(A)=\left\{\begin{array}{cl}
C^n(A),&n\ge 0,\\
D^{n+1}(A),&n<0,
\end{array}\right.
\end{equation}
which calculates Tate cohomology: $\hat H^n(G,A)=H^n(\hat C^\bullet (A)^G)$, $n\in \Z$.

\begin{definition} Let $G$ be a finite group and let $A^\bullet$ be a bounded complex of $G$-modules. We put
\begin{equation*}
\hat \bH(G,A^\bullet)= \Tot(\hat C^\bullet(A^\bullet))^G, \quad
\hat H^n(G,A^\bullet)=H^n(\hat \bH(G,A^\bullet)), \ n\in \Z.
\end{equation*}
\end{definition}

\noindent
For a bounded complex $A^\bullet$ the spectral sequence of the double complex  $\hat C^\bullet(A^\bullet)^G$
\[
E_2^{st}= \hat H^s(G,H^t(A^\bullet))\Rightarrow \hat H^{s+t}(G,A^\bullet)
\]
converges. Hence for an exact bounded complex $A^\bullet$ also the complex $\hat \bH(A^\bullet)$ is exact. Moreover, the assignment $A^\bullet \mapsto \hat \bH(G,A^\bullet)$ commutes with the operation of taking the cone of a complex homomorphism.
We therefore obtain a functor ``Tate cohomology'' on $\cD^b(G\textrm{-Mod})$, the derived category of bounded complexes of $G$-modules. For a single $G$-module $A$ considered as a complex concentrated in degree zero this gives back the usual Tate cohomology. Note that the natural map $C^\bullet(A)\to \hat C^\bullet(A)$ induces a map $\bH(G,A^\bullet) \to \hat \bH(G,A^\bullet)$.

\medskip
If $G=\Gal(\bC|\bR)$, we can translate this definition into the language of sheaves and obtain the modified (hyper) cohomology \begin{equation}
\hat \bH_\et(\Spec \bR, \cF^\bullet), \quad \hat H^n_\et(\Spec \bR, \cF^\bullet)= H^n(\hat \bH_\et(\Spec \bR, \cF^\bullet))
\end{equation}
for any bounded complex $\cF^\bullet$ of sheaves on $(\Spec \bR)_\et$. We have the hypercohomology spectral sequence
\begin{equation}
E_2^{st}=\hat H^s_\et(\Spec \bR, H^t(\cF^\bullet))\Longrightarrow \hat H^{s+t}_\et(\Spec \bR, \cF^\bullet).
\end{equation}

\begin{definition}
Let $f: X\to \Spec \bR$ be a separated scheme of finite type and let $\cF^\bullet$ be a bounded complex of sheaves on $X_\et$. We define the \emph{modified \'{e}tale hypercohomology} by
\begin{equation*}
\hat \bH_\et(X, \cF^\bullet)=\hat \bH_\et(\Spec \bR, Rf_*\cF^\bullet),
\end{equation*}
and put
\begin{equation*}
\hat H^n_\et(X,\cF^\bullet)= H^n(\hat \bH_\et(X, \cF^\bullet))=\hat H^n_\et(\Spec \bR, Rf_*\cF^\bullet),\  n\in \Z.
\end{equation*}
If $\cF^\bullet$ is torsion, we define the \emph{modified \'{e}tale hypercohomology with compact support} by
\begin{equation*}
\hat \bH_c(X, \cF^\bullet)=\hat \bH_\et(\Spec \bR, Rf_!\cF^\bullet),
\end{equation*}
and put
\begin{equation*}
\hat H^n_c(X,\cF^\bullet)= H^n(\hat \bH_c(X, \cF^\bullet))=\hat H^n_c(\Spec \bR, Rf_!\cF^\bullet),\  n\in \Z.
\end{equation*}
\end{definition}

\noindent
Note that this makes sense since $Rf_*\cF^\bullet,Rf_!\cF^\bullet\in \cD^b(\Sh_\et(\Spec \bR))$.

\medskip
We will also modify the Ext-groups. Let $f: X\to \Spec \bR$ be a separated scheme of finite type and $\cF^\bullet$, $\cG^\bullet$ complexes of sheaves on $X_\et$. Recall that
\begin{equation}
R\Hom_{X}(\cF^\bullet,\cG^\bullet)= \bH_\et(X, R\IHom(\cF^\bullet,\cG^\bullet)).
\end{equation}

\begin{definition} If $R\IHom(\cF^\bullet,\cG^\bullet)$ is bounded, we define
\[
\hat{R}\Hom_{X}(\cF^\bullet,\cG^\bullet)= \hat \bH_\et(X, R\IHom(\cF^\bullet,\cG^\bullet))
\]
and put
\[
\widehat{\Ext}^n_{X}(\cF^\bullet,\cG^\bullet)= H^n(\hat{R}\Hom_{X}(\cF^\bullet,\cG^\bullet)),\ n\in \Z.
\]
\end{definition}

\begin{remark}\label{agree}
Assume that $f: X\to \Spec \bR$ factors through $\Spec \bC$, i.e., $f=ig$ with $g: X\to \Spec \bC$ and $i: \Spec \bC \to \Spec \bR$ the canonical morphism. Then
$Rf_*= i_* \circ Rg_*= \Ind^{\{1\}}_{\Gal(\bC|\bR)} \circ Rg_*$ and hence
\begin{equation*}
\hat H^n_\et(X,\cF^\bullet)= \hat H^n_\et(\Gal(\bC|\bR), Rf_*\cF^\bullet)= \hat H^n_\et(\{1\}, Rg_*\cF^\bullet)= 0. \end{equation*}
The same argument works for cohomology with compact support and for Ext-groups.
\end{remark}

\noindent
\emph{Notational convention:}
For a scheme $X$ over~$\bC$,  we put
\begin{equation}\label{conventionC}
\hat H^n_\et(X,\cF^\bullet) = 0.
\end{equation}
For a scheme $X$ over a non-archimedean local field, we put
\begin{equation}\label{convention}
\hat H^n_\et(X,\cF^\bullet) = H^n_\et(X,\cF^\bullet),
\end{equation}
i.e., the $\hat{\text{ }}$ is redundant. The same convention applies to cohomology with compact support and to Ext-groups.

\bigskip
Next we consider the scheme $\Spec \Z$.
The set $\{\Z\to \bC,\  (\Z \to \bar \F_p)_{p \text{ prime}}\}$ of geometric points is conservative and we consider
the associated Godement resolution of a sheaf  $\cF$ on $(\Spec \Z)_\et$:
\begin{equation}\
\cF \lang C^\bullet(\cF).
\end{equation}
We consider the composite
\begin{equation}
\varphi: \Gamma(\Spec \Z, C^\bullet(\cF)) \to \Gamma(\Spec \bR, C^\bullet(\cF|_{\Spec \bR}))\to   \Gamma(\Spec \bR, \hat C^\bullet(\cF|_{\Spec \bR}))
\end{equation}
\begin{definition} We define
\begin{equation*}
\hat \bH_\et(\Spec \Z,\cF)= \text{cone}(\varphi)[-1].
\end{equation*}
For a bounded complex $\cF^\bullet$ of sheaves on $(\Spec \Z)_\et$, we obtain the double complex $\hat \bH_\et(\Spec \Z, \cF^\bullet)$, which we also consider as a single complex via the total complex functor and set
\begin{equation*}
 \hat H^n_\et(\Spec \Z, \cF^\bullet) = H^n(\hat \bH_\et(\Spec \Z, \cF^\bullet)).
\end{equation*}
\end{definition}

\noindent
For a bounded complex $\cF^\bullet$ of sheaves the spectral sequence of the double complex $\hat \bH_\et(\Spec \Z, \cF^\bullet)$
converges. Hence for an exact bounded complex $\cF^\bullet$ the complex $\hat \bH_\et(\Spec \Z, \cF^\bullet)$ is still exact. Moreover, the assignment $\cF^\bullet \mapsto \hat \bH_\et(\Spec \Z, \cF^\bullet)$ commutes (up to a canonical isomorphism of complexes) with the operation of taking the cone.
Therefore  $\hat H^n_\et(\Spec \Z,-)$ is a functor on $\cD^b(\Sh_\et(\Spec \Z))$ and the following definition makes sense.
\begin{definition}
Let $f: \cX\to \Spec \Z$ be a separated scheme of finite type and $\cF^\bullet$ a bounded complex of sheaves on $X_\et$. We define the \emph{modified \'{e}tale cohomology} by
\begin{equation*}
\hat H^n_\et(\cX,\cF^\bullet)= \hat H^n_\et(\Spec \Z, Rf_*\cF^\bullet) , \ n\in \Z.
\end{equation*}
If\/ $\dim \cX=1$ or $\cF^\bullet$ is torsion  (i.e., if we have a well-defined functor $Rf_!$), we define the \emph{modified \'{e}tale cohomology with compact support} by
\begin{equation*}
\hat H^n_c(\cX,\cF^\bullet)= \hat H^n_\et(\Spec \Z, Rf_!\cF^\bullet), \ n\in \Z.
\end{equation*}
\end{definition}

\noindent
For a single sheaf $\cF$ and $n<0$, the groups $\hat H^n_\et(\cX,\cF)\cong \hat H^{n-1}_\et(X_\bR,\cF)$ are $2$-torsion groups. In general, we have a long exact sequence
\begin{equation}\label{compared}
\cdots \to \hat H^n_\et(\cX,\cF^\bullet) \stackrel{\varphi_n}{\to}  H^n_\et(\cX,\cF^\bullet) \to \hat H^n_\et(X_\bR,\cF^\bullet)\to \hat H^{n+1}_\et(\cX,\cF^\bullet)\to \cdots
\end{equation}
The maps $\varphi_n$ are isomorphisms for all $n\in \Z$ if $\cX\to \Spec \Z$ factors through $\Spec \F_p$ for a prime number $p$, or over $\Spec \O_k$ for a totally imaginary number field~$k$ (cf.\ \cref{agree}). The compact support variant of (\ref{compared}) is the long exact sequence
\begin{equation}\label{comparedc}
\cdots \to \hat H^n_c(\cX,\cF^\bullet) \to  H^n_c(\cX,\cF^\bullet) \to \hat H^n_c(X_\bR,\cF^\bullet)\to \hat H^{n+1}_c(\cX,\cF^\bullet)\to \cdots
\end{equation}
If $\cX= \Spec \O_S$ for a number field $k$ and a finite set $S\supset S_\infty$ of places of $k$, and if $\cF$ is a single sheaf on $\cX_\et$, then our groups $\hat H^*_\et(\cX,\cF)$ coincide with the modified \'{e}tale cohomology groups defined in \cite{Zi}.

\bigskip
Now let $k$ be a global field. For a place $v$ of $k$, we denote by $k_{(v)}$ the henselization of $k$ at $v$. For a set $S$ of places of $k$ we set $\cS= \Spec \O_S$.  If $p=\ch(k)>0$  we make the conventions that

\medskip
\begin{compactitem}
\item $S_\infty=\varnothing$.
\item \lq{$\Spec \O_\varnothing$}\rq\ is the unique smooth, proper curve with function field $k$.
\item cohomology with compact support is defined with respect to the structure morphism to $\Spec \bF_p$.
\item the modification symbol $\hat{\text{ }}$ is redundant.
\end{compactitem}

\medskip\noindent
For sets of places $S\supset T \supset S_\infty$ and $\cT= \Spec \O_T$, we put
\begin{equation}
L^n_T(S, \cF^\bullet)= \bigoplus_{v\in T} \hat H^n(k_{(v)}, \cF^\bullet)  \oplus \bigoplus_{v\in S\smallsetminus T} H^{n+1}_v(\cT,\cF^\bullet).
\end{equation}

The following lemma compares the modified compact support cohomology with ordinary cohomology. It is interesting even for  $S=T$.

\begin{lemma}\label{verglohnec}  Let $S\supset T\supset S_\infty$ be sets of places with $T$ finite and $\cF^\bullet$ a bounded complex of sheaves on $\cT_\et$. Then we have a long exact sequence
\begin{equation*}
\cdots \to \hat H^n_c(\cT,\cF^\bullet) \to H^n_\et(\cS,\cF^\bullet|_\cS) \to L^n_T( S, \cF^\bullet)     \to  \hat H^{n+1}_c(\cT,\cF^\bullet)\to \cdots .
\end{equation*}
\end{lemma}

\begin{proof} We can assume that $S$ is finite since the general case follows by passing to the direct limit over all finite $S'$ with $S\supset S'\supset T$.
Put $\cB=\Spec \O_{\varnothing}$ and let $j: \cT \hookrightarrow \cB$ be the open immersion.
Consider the commutative diagram
\[
\begin{tikzcd}
\cone(\psi)[-1]\rar\dar&\hat \bH_c(\cT,\cF^\bullet)\rar{\psi}\dar& \bH(\cS,\cF^\bullet)\arrow[d,equal]\\
\bH_S(\cB,j_!\cF^\bullet)\rar\dar{0}&\bH_c(\cT,\cF^\bullet)\rar\dar& \bH(\cS,\cF^\bullet)\dar\\
\hat \bH(\cB_\bR,\cF^\bullet)\rar&\hat \bH(\cB_\bR,\cF^\bullet)\rar&\,0.
\end{tikzcd}
\]
The first and third row, as well as the second and third column are distinguished triangles by definition. Also the second row is distinguished: it is the excision triangle for $\cS\subset \cB$ and the complex $j_!\cF^\bullet$. We conclude that the first column is distinguished. The lower left vertical arrow is zero since it factors through $\bH(\cS,\cF^\bullet)$. Hence
\begin{equation}\label{triaplus}
\cone(\psi) \cong \bH_S(\cB,j_!\cF^\bullet)[1] \oplus \hat \bH(\cB_\bR,\cF^\bullet).
\end{equation}
Finally note that
\begin{equation}\label{222}
\bH_S(\cB,j_!\cF^\bullet)\cong \bigoplus_{v\in S\smallsetminus S_\infty} \bH_{v}(\cB,j_!\cF^\bullet)
\end{equation}
and that
\begin{equation}\label{221}
\bH_{v}(\cB, j_!\cF^\bullet) \cong \bH(k_{(v)},\cF^\bullet)[-1]
\end{equation}
for $v\in T\smallsetminus S_\infty$.
\end{proof}

\section{The dualizing complex}\label{sec4}
In this section let $B$ be an integral, noetherian and regular scheme of Krull-di\-men\-sion~$\leq 1$.

\begin{definition}\label{dimdef}
For
$f: X\to B$ integral, separated and of finite type, we put $Y=\overline{f(X)}\subset B$ and  define the (absolute) \emph{dimension of\/ $X$} by
\begin{equation*}
\dim X = \hbox{tr.deg}(k(X)/k(Y)) - \codim_B(Y) +\dim_{\text{Krull}} B.
\end{equation*}
\end{definition}

\noindent
The dimension of $X$ coincides with its Krull-dimension~if $X\to B$ is proper or if $\dim B=0$ or if $\dim B=1$ and $B$ has infinitely many closed points. It differs, for example, for $X=\Spec \Q_p$ considered as a scheme over $B=\Spec \Z_p$.

\medskip\noindent
For a scheme $X$, separated and of finite type over $B$ and an integer $n$, let $\Z^c_X(n)$ be Bloch's cycle complex of relative dimension $n$
considered in \cite{Ge}. It is the bounded above complex of sheaves on the small \'{e}tale site of $X$ such that for any \'{e}tale $W\to X$ we have
\begin{equation}
\Z^c_X(n)^i(W)= z_n(W,-i-2n).
\end{equation}
Here $z_q(W,p)$ is the free abelian group generated by integral $(p+q)$-dimensional subschemes of  $\Delta^{p}_W$ that intersect all faces properly. We list some properties of $\Z^c_X(n)$.

\begin{lemma}\label{zc=gm}
If $\dim B=0$, there are natural quasi-isomorphisms
\[
\Z_B^c(0)\cong \Z, \quad \Z^c_B(-1)\cong \G_m[-1].
\]
If\/ $\dim B=1$, there is a natural quasi-isomorphism
$\Z^c_B(0)\cong \G_m[1]$.
\end{lemma}

\begin{proof} If $B$ is a field, then $\Z_B^c(0)\cong \Z$ follows directly from the definition and $\Z^c_B(-1)\cong \G_m[-1]$ follows from \cite[Cor.\ 6.4]{Bl}. If $B$ is a regular curve over a field, then $\Z^c_B(0)\cong \G_m[1]$ follows again from \cite[Cor.\ 6.4]{Bl}. For the general one-dimensional case see \cite[Lemma 11.2]{Le} or \cite[Cor.\ 3.9]{Sch}.
\end{proof}

The definition of $\Z^c_X(n)$ naturally extends to schemes which are filtered inverse limits of \'{e}tale $X$-schemes (e.g., the strict henselization of $X$ at a geometric point). We will call such schemes \'{e}tale essentially of finite type. The dimension of  $W=\varprojlim W_i$ is defined as the common dimension of the $W_i$. We have
\begin{equation}\label{limcycle}
z_q(\varprojlim_i W_i, p)= \varinjlim_i z_q(W_i, q),
\end{equation}
If $\dim B=1$, $K=k(B)$ and $X\to B$ factors through $\Spec K$, then our definition of dimension implies
\begin{equation} \label{twistversch}
\Z^c_{X /B}(n) = \Z^c_{X/K} (n-1)[2].
\end{equation}

\begin{notconv}\label{betaconv}
In the following, we assume that $\beta$ is a natural number such that $k(b)$ is imperfect at most for points $b\in B$ with $\ch k(b)\mid \beta$.
\end{notconv}

If $B$ is the spectrum of the ring of integers of a number field or a $p$-adic field, we can put $\beta=1$. If $\ch (k(B))=p>0$, we can put $\beta=p$.

\begin{theorem}[Localization] \label{localization}
Let $n\le 0$ and $i: Z\hookrightarrow X$ a closed embedding with open complement $j:U\to X$. Then there is a natural isomorphism
\[
Ri^!\Z^c_X(n) \otimes \Z[\tfrac{1}{\beta}]\cong \Z^c_Z(n)\otimes \Z[\tfrac{1}{\beta}]
\]
in the derived category of sheaves on~$Z_\et$. In particular, we have a distinguished triangle
\begin{equation*}
 i_*\Z^c_Z(n) \otimes \Z[\tfrac{1}{\beta}] \to \Z^c_X(n)\otimes \Z[\tfrac{1}{\beta}] \to Rj_*\Z^c_U(n)\otimes \Z[\tfrac{1}{\beta}] \stackrel{[1]}{\to}
\end{equation*}
in the derived category of sheaves on $X_\et$.
\end{theorem}

\begin{proof}
\cite[Cor.\ 7.2 a)]{Ge} applies with the same proof after tensoring with $\Z[\tfrac{1}{\beta}]$.
\end{proof}

\begin{theorem}[Duality]\label{dualgl} Let $f:X\to B$ be separated and of finite type and let $\cF^\bullet$ be a bounded torsion complex of  sheaves on $X_\et$.
Then, for $n\leq 0$, we have a natural isomorphism in the derived category of abelian groups
\[
R\Hom_X(\cF^\bullet,\Z^c_X(n)\otimes \Z[\tfrac{1}{\beta}] ) \cong R\Hom_B(Rf_! \cF^\bullet, \Z^c_B(n)\otimes \Z[\tfrac{1}{\beta}]).
\]
\end{theorem}
\begin{proof} \cite[Cor.\ 4.7 b), Thm.\ 7.3]{Ge}  shows that
\begin{equation}
R\Hom_X(\cF,\Z^c_X (n)) \cong R\Hom_B(Rf_! \cF, \Z^c_B(n))
\end{equation}
for a single torsion sheaf $\cF$, if $B$ is the spectrum of a perfect field or of a Dedekind domain of characteristic zero with perfect residue fields. Without the perfectness assumption, the same proof shows  the isomorphism for general $B$ after inverting~$\beta$. Finally, this extends in a straightforward manner to the case of a bounded complex~$\cF^\bullet$.
\end{proof}

\begin{corollary}[Modified duality]\label{modifieddual}  Assume that $B=\Spec \bR$ in \cref{dualgl}.
Then, for $n\leq 0$, we have a natural isomorphism in the derived category of abelian groups
\begin{equation*}
\hat{R}\Hom_X(\cF^\bullet,\Z^c_X(n) ) \cong \hat{R}\Hom_{\Spec \bR}(Rf_! \cF^\bullet, \Z^c_{\Spec \bR}(n)).
\end{equation*}
\end{corollary}

\begin{proof}
Let $X_\bC= X\times_\bR \bC$.  The isomorphism
\begin{equation}\label{dualglueberc}
R\Hom_{X_\bC}(\cF^\bullet|_{X_\bC},\Z^c_{X_\bC}(n) ) \cong R\Hom_{\Spec \bC}(Rf_! \cF^\bullet|_{X_\bC}, \Z^c_{\Spec \bC}(n))
\end{equation}
of \cref{dualgl} for $n\le 0$ is $\Gal(\bC|\bR)$-invariant (cf.\ the proof of \cite[Thm.\ 4.1)]{Ge}). By \cite[Lemma 4.8]{Ge} the complexes in (\ref{dualglueberc}) are bounded. Hence we can apply $\hat\bH(\Gal(\bC|\bR),-)$ to obtain the result.
\end{proof}

Working with sheaves of $\Z/m\Z$-modules, we obtain a dimension shift by one by the following lemma.

\begin{lemma}\label{wechselauftorsion}
Let $\cF^\bullet$ be a bounded complex of sheaves of\/ $\Z/m\Z$-modules on~$X_\et$. Then we have a natural isomorphism in the derived category of abelian groups
\begin{equation*}
R\Hom_{X,\Z/m\Z}(\cF^\bullet, \Z^c_X(n)/m)[-1]\cong R\Hom_X(\cF^\bullet, \Z^c_X(n)).
\end{equation*}
\end{lemma}

\begin{proof}
This is standard homological algebra and has nothing to do with algebraic cycle complexes in particular. The proof of \cite[Lemma 2.4]{Ge} applies without change.
\end{proof}

Let $X$ be a regular scheme which is \'{e}tale essentially of finite type over $B$ and equidimensional of dimension~$d$. Let $m$ be a natural number invertible on $B$. Then, for $\dim B=0$, the \'{e}tale cycle class map $c_X$ is defined in \cite{GL}. It is a natural homomorphism
\begin{equation}\label{cycle1}
c_X: \Z^c_X (n)/m \lang \mu_m^{\otimes d-n}[2d]
\end{equation}
of complexes of sheaves on $X_\et$. We recall:
\begin{theorem}[{\cite[Thm.\ 1.5]{GL}}]\label{identifyZc0} If $\dim B=0$ and $X$ is regular, then $c_X$ is a quasi-isomorphism for $n\leq d$.
\end{theorem}
The construction of $c_X$ extends to the case $\dim B=1$ as explained in \cite[\S12.3]{Le}.
If $X\to B$ factors through $\Spec K$, $K=k(B)$, then the cycle class maps over $B$ and over $K$ are compatible, i.e., the diagram
\begin{equation}
\begin{tikzcd}[column sep=large]
\Z^c_{X/B}(n)/m\rar{c_{X\to B}}\arrow[d,"="]& \mu_m^{\otimes d-n }[2d]\arrow[d,"="]\\
\Z^c_{X/K}(n-1)/m\,[2]\rar{c_{X\to K}[2] }& \mu_m^{\otimes d-n} [2d]
\end{tikzcd}
\end{equation}
commutes (note that $X$ has dimension $d-1$ as a scheme over $K$).
\begin{theorem}\label{identifyZc} Assume $\dim B=1$ and that $X$ is \'{e}tale essentially of finite type over~$B$, regular connected and of dimension~$d$. Let $m$ be a natural number invertible on $X$. Then the cycle class map
\begin{equation*}
c_X: \Z^c_X(n)/m \lang \mu_m^{\otimes d-n}[2d].
\end{equation*}
is a quasi-isomorphism for $n\le 0$.
\end{theorem}

\begin{remark}
One expects that $c_X$ is a quasi-isomorphism for $n\le d$.
\end{remark}

\begin{proof}[Proof of \cref{identifyZc}] The proof follows \cite[Thm.\ 12.5]{Le}. We can assume that $B$ is local and strictly henselian.
Let $Y$ be essentially of finite type over $B$, $i: Z\to Y$ be a regular
closed subscheme and $j: W\to Y$ its open complement.  We have a map
of distinguished triangles of complexes of sheaves on~$Y_\et$,
\begin{equation}
\begin{tikzcd}
i_* \Z_Z^c(n)/m \rar\dar{c_Z} & \Z^c(n)/m \rar\dar{c_Y} & Rj_* \Z^c_W(n)/m\dar{c_W} \\
i_* \mu_n^{\otimes d-n}[2d]_Z \rar& \mu_n^{\otimes d-n}[2d]
 \rar & Rj_*\mu_m^{\otimes d-n}[2d]_W.
\end{tikzcd}
\end{equation}
Indeed, the upper row is distinguished by \cref{localization},
the lower row by purity of \'{e}tale cohomology \cite{gabber}, and
the commutativity of the diagram follows from the definition of
the cycle class map, see \cite[\S12.3]{Le}.

We see that if the result holds for two of $W,Z,Y$, then it holds
for the third. The special fibre of $Y$ has a
finite stratification by closed subsets $T_i$ such that $T_i-T_{i+1}$
is regular. Thus the result follows by induction on $i$ from \cref{identifyZc0}.
\end{proof}

The following lemma will be used in \cref{proofsect}.
\begin{lemma}\label{extdegen}
Let $\cF$ be a locally constant, constructible sheaf of\/ $\Z/m\Z$-modules on~$X_\et$, where $m$ is invertible on $X$. Then, for $n\ge 0$, we have  natural isomorphisms  in the derived category of abelian groups
\begin{equation*}
R\Hom_{X,\Z/m\Z}(\cF, \mu_m^{\otimes n})\cong R\Gamma(X,\IHom_{X,\Z/m\Z}(\cF, \mu_m^{\otimes n})).
\end{equation*}
\end{lemma}

\begin{proof}
We consider the local-to-global spectral sequence
\[
E_2^{ij} = H^i_\et(X,\IExt^j_{X,\Z/m\Z}(\cF,\mu_m^{\otimes n}))\Rightarrow \Ext^{i+j}_{X,\Z/m\Z}(\cF,\mu_m^{\otimes n}).
\]
By \cite[III,\,Ex.\,1.39\,(b)]{Mi}, we have for any $x\in X$ the isomorphisms of stalks
\[
\IExt^j_{X,\Z/m\Z}(\cF,\mu_m^{\otimes n})_x \cong \Ext^{j}_{\Z/m\Z}(\cF_x,(\mu_m^{\otimes n})_x).
\]
Since $(\mu_m^{\otimes n})_x\cong \Z/m\Z$ is an injective $\Z/m\Z$-module, we obtain
\[
\IExt^j_{X,\Z/m\Z}(\cF,\mu_m^{\otimes n})=0, \quad j \ge 1.
\]
Hence the spectral sequence degenerates
showing the statement of the lemma.
\end{proof}

\section{Arithmetic duality for complexes}
Using the results of \cref{sec4}, we generalize local and global duality theorems for single sheaves to bounded complexes and higher dimensions. For technical reasons, we also have to work with henselian local fields in the following sense.

\begin{definition}\label{henseliandef}
In this paper a \emph{henselian local field}\/ is one of the following:
\begin{itemize}
  \item[{(nonarchimedean):}]the quotient field of an excellent, henselian, discrete valuation ring with finite residue field.
  \item[{(complex)}:]an algebraically closed subfield of $\bC$.
  \item[(real):]a relatively algebraically closed subfield of $\bR$.
\end{itemize}
\end{definition}

\begin{remark} Complete discrete valuation rings are excellent by \cite[7.8.3\,(iii)]{ega4}. The local rings of global fields and their henselizations are excellent by \cite[7.8.3\,(ii)]{ega4} and \cite[18.7.6]{ega4}.
\end{remark}
\noindent
The henselian local fields of \cref{henseliandef} come with a natural valuation and their completions are the local fields in the usual sense.
For a henselian local field $K$ with completion $\widehat K$,  the natural homomorphism of absolute Galois groups
$ \mathit{Gal}_{\widehat K}\to\mathit{Gal}_K$ is an isomorphism by \cite[X, 2.2.1]{sga4}. Hence the known statements about the Galois cohomology of local fields immediately extend to henselian local fields in the sense of \cref{henseliandef}. Also the statement of  \cref{modifieddual} holds for arbitrary real henselian local fields.

\begin{theorem}[Duality over henselian local fields] \label{Duality over local fields} Let $K$ be a henselian local field and let $f:X\to \Spec K$ be separated and of finite type.  Let $\cF^\bullet$ be a bounded, constructible complex of sheaves on~$X_\et$.  If $\ch(K)=p>0$ assume that $\cF^\bullet$ is $p$-torsion free.
Then Tate's local duality  induces perfect pairings of finite abelian groups
\[
\hat H^r_c(X, \cF^\bullet) \times \widehat{\Ext}_{X}^{3-r}(\cF^\bullet,\Z^c_{X}(-1))\lang \Q/\Z
\]
for all $r\in \Z$.
\end{theorem}

\begin{proof} We start with the case $X=\Spec K$ and a single sheaf $\cF$. By Tate's local duality theorem \cite[(7.2.6), (7.2.17)]{NSW}, the cup product followed by the invariant map of local class field theory induce a perfect pairing of finite abelian groups
\begin{equation}\label{localpairingdim0}
\hat H^r(K, \cF) \times \hat H^{2-r}(K, \IHom(\cF,\G_m))\stackrel{\cup}{\lang} H^2(K,\G_m) \stackrel{\inv}{\lang}   \Q/\Z.
\end{equation}
for all $r\in \Z$.
Since $\IHom_{K}(-,\G_m)$ is exact on prime-to-$p$ torsion sheaves, we have $\mathcal{E}\!xt^i_{K}(\cF,\G_m)=0$ for $i>0$ and the local to global spectral sequence implies that
\begin{equation}
\hat H^{2-r}(K, \IHom(\cF,\G_m))\cong \widehat{\Ext}_{K}^{2-r}(\cF,\G_m).
\end{equation}
By \cref{zc=gm}, we have $\Z^c_{K}(-1)\cong \G_m[-1]$. This shows the perfect pairing of the theorem for a single sheaf and the result for a bounded complex $\cF^\bullet$ follows from the map of hypercohomology spectral sequences
\[
\begin{tikzcd}
H^s_\et(K,H^t(\cF^\bullet))\arrow[r,Rightarrow]\arrow[d]&H_\et^{s+t}(K,\cF^\bullet)\dar\\
\widehat{\Ext}_{K}^{3-s}(H^t(\cF^\bullet),\Z^c_K(-1))^\vee\arrow[r,Rightarrow]& \widehat{\Ext}_{K}^{3-s-t}(\cF^\bullet,\Z^c_K(-1))^\vee.
\end{tikzcd}
\]
In the general case, we have
\begin{equation}
\hat H^r_c(X, \cF^\bullet)= \hat H^r_c(K, Rf_!\cF^\bullet)
\end{equation}
by definition. Furthermore,
\begin{equation}\widehat{\Ext}_{X}^{3-r}(\cF^\bullet,\Z^c_{X}(-1))\cong \widehat{\Ext}_{K}^{3-r}(Rf_! \cF^\bullet,\Z^c_{K}(-1))
\end{equation}
by \cref{dualgl} and \cref{modifieddual}.  Hence we obtain the asserted natural perfect pairing from the diagram
\[
\begin{tikzcd}[column sep=.2cm]
\hat H^r_c(X, \cF^\bullet)\arrow[d,"\wr"]&\times& \widehat{\Ext}_{X}^{3-r}(\cF^\bullet,\Z^c_{X}(-1))\arrow[d,"\wr"] \\
\hat H^r_c(K, Rf_!\cF^\bullet)&\times& \widehat{\Ext}_{K}^{3-r}(Rf_! \cF^\bullet,\Z^c_{K}(-1))\arrow[rr,"\cup"]&& H^2(K,\G_m) \arrow[rr,"\inv"]&&\Q/\Z.
\end{tikzcd}
\]
\end{proof}

\begin{theorem}[Henselian local duality]\label{Local duality}
Let $K$ be a nonarchimedean henselian local field, $B=\Spec \O_K$, $b\in B$ the closed point and $f: \cX\to B$ separated and of finite type. Let $\cF^\bullet$ be a bounded, constructible complex of sheaves on $\cX_\et$.  If $\ch(K)=p>0$ assume that $\cF^\bullet$ is $p$-torsion free.
Then the local duality on $B$ induces  perfect pairings of finite abelian groups
\[
H^r_{\{b\}}(B, Rf_!\cF^\bullet) \times \Ext_{\cX}^{2-r}(\cF^\bullet,\Z^c_{\cX}(0))\lang \Q/\Z.
\]
for all $r\in \Z$.
\end{theorem}
\begin{remark}
The group $H^r_{\{b\}}(B, Rf_!\cF^\bullet)$ is the cohomology of $\cX$ with compact support in the closed fibre.
\end{remark}
\begin{proof}[Proof of \cref{Local duality}] We start with the case $\cX=B$ and a single sheaf $\cF$. Then the local duality theorem \cite[II, Thm.\ 1.8(b)]{ADT} states that
\begin{equation}\label{dim1localdual}
H^r_{\{b\}}(B,\cF) \times \Ext^{3-r}_{B}(\cF,\G_m) \stackrel{\cup}{\lang} H^3_{\{b\}}(B,\G_m) \stackrel{\tr}{\lang} \Q/\Z
\end{equation}
is a perfect pairing of finite abelian groups for all $r\in \Z$. By \cref{zc=gm}, we have $\Z^c_{B}\cong \G_m[1]$.
This shows the perfect pairing of the theorem for a single sheaf and the result for a bounded complex $\cF^\bullet$ follows from the hypercohomology spectral sequence.

In the general case,  we obtain the asserted perfect pairing from the diagram
\[
\begin{tikzcd}[column sep=.2cm]
H^r_{\{b\}}(B, Rf_!\cF^\bullet)\arrow[d,"="]&{\times}& \Ext_{\cX}^{2-r}(\cF^\bullet,\Z^c_{\cX}(0))\arrow[d,"\wr"] \\
H^r_{\{b\}}(B,Rf_!\cF^\bullet)&{\times}& \Ext^{2-r}_{B}(Rf_!\cF^\bullet,\Z^c_{B}(0)) \arrow[rr]&&\Q/\Z.
\end{tikzcd}
\]
\end{proof}

\bigskip
Next we recall the construction of the trace isomorphism (cf.\ \cite[2.5.9]{Zi},  \cite[\S2]{Maz} or \cite[II, 2.6]{ADT}).
For $\cF=\bG_m$ and $S=$(all places),  \cref{verglohnec} and the fact that $H^{n+1}_v(\cT,\bG_m)\cong H^n(k_{(v)},\bG_m)$ for $n \ge 1$  yield an exact sequence
\begin{equation}
0\to \Br(k) \to \bigoplus_{\text{all }v} \Br(k_{(v)}) \to \hat H^3_c(\cT,\bG_m) \to H^3(k, \bG_m)=0.
\end{equation}
Therefore the classical Hasse principle for the Brauer group \cite[(8.1.17)]{NSW} implies the existence of a natural trace isomorphism
\begin{equation}
\tr: \hat H^3_c(\cT,\bG_m) \liso \Q/\Z.
\end{equation}

\begin{theorem}[Generalized Artin-Verdier duality] \label{artinverdier}
Let $k$ be a global field and $\cB=\Spec \O_\varnothing$.
Let $f: \cX\to \cB$ be a separated scheme of finite type and $\cF^\bullet$ a bounded, constructible complex of sheaves on~$\cX_\et$.  If $\ch(k)=p>0$ assume that $\cF^\bullet$ is $p$-torsion free.  Then Artin-Verdier duality induces  perfect pairings of finite abelian groups
\begin{equation*}
\hat H^r_c(\cX, \cF^\bullet) \times \Ext_\cX^{2-r}(\cF^\bullet, \Z^c_\cX(0)) \lang \Q/\Z
\end{equation*}
for all $r\in \Z$.
\end{theorem}

\begin{proof}
If $\cX=\cB$, we have $\Z^c_\cB(0)=\G_m[1]$ by \cref{zc=gm}. In this case, the result follows by the hypercohomology spectral sequence from the classical Artin-Verdier duality for a single constructible sheaf $\cF$, see \cite[Thm.\ 3.2.1]{Zi},  \cite{Maz} and \cite[II, Thm.\ 3.1]{ADT},  resp.\ \cite[II, Thm.\ 6.2]{ADT} if $\ch(k)=p>0$.

The general case then follows from
\begin{equation}
\hat H^r_c(\cX, \cF^\bullet)= \hat H^r(\cB, Rf_! \cF^\bullet)
\end{equation}
by definition, from
\begin{equation}
\Ext_\cX^{2-r}(\cF^\bullet, \Z^c_\cX(0))\cong \Ext_\cB^{2-r}(Rf_! \cF^\bullet, \Z^c_\cB(0))
\end{equation}
by  \cref{dualgl}, and from the fact that  $Rf_! \cF^\bullet$ is a bounded, constructible complex  \cite{fin}.
\end{proof}

\section{A base change property}

\begin{proposition} \label{extnummer2} Let $B$  be an integral, noetherian and regular scheme of Krull-di\-men\-sion~$\le1$, and let $\beta$ be a natural number as in \cref{betaconv}. Let $(B_i)_{i\in I}$ be a filtered inverse system of \'{e}tale $B$-schemes and $B_\infty=\varprojlim B_i$.  Let $f: \cX\to B$ be separated and of finite type and let
$\cF^\bullet$ be a bounded, constructible complex of sheaves on $\cX_\et$.   Consider the fibre product diagram
\[
\begin{tikzcd}
\cX_\infty\rar{\iota_\cX}\arrow[d,"f_\infty"'] &\cX\dar{f}\\
B_\infty\rar{\iota}&B.
\end{tikzcd}
\]
Then, for $n\leq 0$, the natural map
\[
\iota^* Rf_*R\IHom_{\cX}(\cF^\bullet,\Z^c_{\cX}(n)\otimes \Z[\tfrac{1}{\beta}])\lang Rf_{\infty *}R\IHom_{\cX_\infty}(\iota_\cX^*\cF^\bullet,\Z^c_{\cX_\infty}(n)\otimes\Z[\tfrac{1}{\beta}])
\]
is an isomorphism in the derived category of sheaves on $(B_\infty)_\et$.
\end{proposition}

\begin{proof} The natural map of the proposition is the composition of the following three maps:
the first is the base change map of the fibre product diagram
\begin{equation*}
\iota^* Rf_*R\IHom_{\cX}(\cF^\bullet,\Z^c_{\cX}(n)\otimes \Z[\tfrac{1}{\beta}])\to Rf_{\infty *}\iota_{\cX}^*R\IHom_{\cX} (\cF^\bullet,\Z^c_{\cX}(n)\otimes \Z[\tfrac{1}{\beta}]).
\end{equation*}
The second is the map
\begin{multline*}
Rf_{\infty *}\iota_{\cX}^*R\IHom_{\cX} (\cF^\bullet,\Z^c_{\cX}(n)\otimes \Z[\tfrac{1}{\beta}]) \to \\
Rf_{\infty *}R\IHom_{\cX_\infty} (\iota_{\cX}^* \cF^\bullet,\iota_{\cX}^*\Z^c_{\cX}(n)\otimes \Z[\tfrac{1}{\beta}])
\end{multline*}
induced by  $\iota_{\cX}^*R\IHom_{\cX}(-,-) \to R\IHom_{\cX_\infty}(\iota_{\cX}^*- , \iota_{\cX}^*-)$,
and the third is the map
\begin{multline*}
Rf_{\infty *} R\IHom_{\cX_\infty} (\iota_{\cX}^* \cF^\bullet,\iota_{\cX}^*\Z^c_{\cX}(n)\otimes \Z[\tfrac{1}{\beta}]) \to   \\
Rf_{\infty *} R\IHom_{\cX_\infty} (\iota_{\cX}^* \cF^\bullet,\Z^c_{\cX_\infty}(n)\otimes \Z[\tfrac{1}{\beta}])
\end{multline*}
induced by the natural map $\iota_{\cX}^*\Z^c_{\cX}(n) \to \Z^c_{\cX_\infty}(n)$.
By \cref{dualgl}, it suffices to show that the map
\[
\iota^* R\IHom_{B}(Rf_! \cF^\bullet,\Z^c_{B}(n)\otimes\Z[\tfrac{1}{\beta}])
\lang R\IHom_{B_\infty}({Rf_\infty}_! \iota_\cX^*\cF^\bullet,\Z^c_{B_\infty}(n)\otimes\Z[\tfrac{1}{\beta}])
\]
is an isomorphism.
By \cite[XVII, 5.2.6]{sga4}, $Rf_!$ commutes with base change and sends bounded, constructible complexes to bounded, constructible complexes by \cite[XVII, 5.3.6]{sga4}. Hence we have $ R{f_\infty}_! \iota_\cX^*\cF^\bullet\cong \iota^* Rf_! \cF^\bullet$ and it suffices to show that the map
\begin{equation}\label{extlocalis}
\iota^* R\IHom_{B}(\cF^\bullet,\Z^c_{B}(n)\otimes\Z[\tfrac{1}{\beta}])
\lang R\IHom_{B_\infty}(\iota^*\cF^\bullet,\Z^c_{B_\infty}(n)\otimes\Z[\tfrac{1}{\beta}])
\end{equation}
is an isomorphism for every bounded, constructible complex $\cF^\bullet$ on $B$. By the hypercohomology  spectral sequence, we can assume that $\cF^\bullet$ is a single sheaf $\cF$.

A geometric point $\bar x\to B_\infty$ induces a compatible system of geometric points $\bar x\to B_i\to B$. Assume for the moment that $\cF$ is locally constant. Then by \cite[III,\,Exercise\,1.31\,b)]{Mi} and \cref{limcycle}, the map induced by (\ref{extlocalis}) on the stalks at $\bar x$ is the identity of
\[
R\Hom_{\Z}(\cF_{\bar x},\Z^c_{B}(n)_{\bar x}\otimes\Z[\tfrac{1}{\beta}]).
\]
This shows that (\ref{extlocalis}) is an isomorphism if $\cF$ is locally constant. For a general constructible $\cF$, there is a non-empty open $j:U\to B$ such that $\cF|_U$ is locally constant.    Let $i:Z\to B$ be the closed complement. Then $Z$ has dimension zero, hence $i^*\cF$ is locally constant. Using the short exact sequence $0\to j_!j^*\cF \to \cF \to i_*i^*\cF\to 0$, the results for $j^*\cF$ and $i^*\cF$, the adjunctions $i_*\dashv i^!$, $j_!\dashv j^*$, and \cref{localization} imply the result for $\cF$.
\end{proof}

\begin{remark}
We will apply \cref{extnummer2} in the following situations:
\begin{itemize}
  \item $B_\infty= \Spec \O_S$, $S$ an infinite set of primes.
  \item $B_\infty= \Spec \O_{(v)}$, the henselization.
\end{itemize}
\end{remark}

\section{Some restricted products}
We recall the notion of a restricted product of topological groups: Assume we are given a family $(A_i)_{i\in I}$ of Hausdorff abelian groups and let an open subgroup $B_i\subset A_i$ be given for almost all $i$. For consistency of notation, we put $B_i=A_i$ for the
remaining indices. The \emph{restricted product} ${\resprod}_{i \in I} (A_i,B_i)$
is the subgroup of\/ ${\prod} A_i$ consisting of all $(x_i)_{i\in I}$
 such that $x_i \in B_i$ for almost all~$i$.
It becomes a topological group by defining  the products
$
\prod_{i \in J}{U_i} \times \prod_{i \in I \smallsetminus J} B_i,
$
where $J$ runs through the finite subsets of\/ $I$ and\/ $U_i$ runs through a
basis of neighbourhoods of the identity of $A_i$ for $i\in J$ as a basis of
neighbourhoods of the identity. Note that
\[
\resprod_{i \in I} (A_i,B_i) \cong \varinjlim_{\substack{J\subset I\\\text{finite}}} \ \left( \prod_{i \in J}{A_i} \ \times  \prod_{i \in I \smallsetminus J} B_i\right),
\]
both in the algebraic and the topological sense, i.e., the direct limit topology on the right hand side is a group topology and coincides with the given topology on the left hand side. Also note that the transition maps in the direct system are injective and open. In particular, for finite $J\subset I$, the groups $\prod_{i \in J}{A_i} \ \times  \prod_{i \in I \smallsetminus J} B_i$ are open subgroups of the restricted product.

\bigskip\noindent
Let $k$ be a global field, $X/k$  separated and  of finite type, $m\ge 1$ an integer prime to $p=\ch (k)$ and $\cF^\bullet$ a bounded, constructible complex of sheaves of $\Z/m\Z$-modules on $X_\et$.
We choose a separated scheme of finite type $f: \cX\to \cB=\Spec \O_\varnothing$ extending $X$, i.e., $X=\cX\times_\cB k$. Choosing $\cX$ small enough, we can assume that $\cF^\bullet$ is the restriction of a bounded, constructible complex of sheaves of $\Z/m\Z$-modules on $\cX_\et$ to $X_\et$. By abuse of notation, we will denote this complex also by~$\cF^\bullet$.
For non\-archi\-me\-dean~$v$, we define the \emph{unramified part}
\begin{equation}
\Ext^n_{X_{(v)},\nr}(\cF^\bullet, \Z^c_{X_{(v)}}(-1))\subset \Ext^n_{X_{(v)}}(\cF^\bullet, \Z^c_{X_{(v)}}(-1))
\end{equation}
as the image of the restriction map
\begin{multline}
\Ext^{n-2}_{\cX_{(v)}}(\cF^\bullet, \Z^c_{\cX_{(v)}}(0))\to\\
\Ext^{n-2}_{X_{(v)}}(\cF^\bullet, \Z^c_{X_{(v)}/\cB_v}(0))\stackrel{\text{Eq.}\,(\ref{twistversch})}{=}  \Ext^n_{X_{(v)}}(\cF^\bullet, \Z^c_{X_{(v)}}(-1)).
\end{multline}

\begin{definition} \label{QSdef}
For a set of primes $S$ we let
\begin{equation}
Q^n_S(X, \cF^\bullet ):=\resprod_{v \in S} \widehat{\Ext}^n_{X_{(v)}}(\cF^\bullet, \Z^c_{X_{(v)}}(-1))
\end{equation}
be the restricted product with respect to the unramified parts (defined for  nonarchimedean~$v$).
We equip the factors $\widehat{\Ext}^n_{X_{(v)}}(\cF^\bullet, \Z^c_{X_{(v)}}(-1))$ with the discrete topology and $Q^n_S(X, \cF^\bullet )$ with the restricted product topology.
\end{definition}

\begin{lemma}
\begin{compactitem}
\item[\rm (i)] $Q^n_S(X, \cF^\bullet )$ is independent of the choice of~$\cX$.
\item[\rm (ii)] $Q^n_S(X, \cF^\bullet )$ is a locally compact topological group.
\end{compactitem}
\end{lemma}

\begin{proof}
(i): If $\cX_1$ and $\cX_2$ are separated schemes of finite type over $\cB$ with generic fibre $X$, then there exists a nonempty open subscheme $\cU \subset \cB$ and an isomorphism $\cX_1|_\cU \liso \cX_2|_\cU$ extending the identity of $X$. Shrinking $\cU$ once more, we may assume that this isomorphism identifies the chosen extensions of the constructible complex $\cF^\bullet$. Hence the subgroups $\Ext^n_{X_{(v)},\nr}(\cF^\bullet, \Z^c_{X_{(v)}}(-1))$ defined with respect to $\cX_1$ and $\cX_2$ coincide for all $v\in \cU$, i.e., for almost all $v$, and the restricted products are the same. \\
(ii): The factors $\widehat{\Ext}^n_{X_{(v)}}(\cF^\bullet, \Z^c_{X_{(v)}}(-1))$ are finite by \cref{Local duality}, hence compact with the discrete topology.  Therefore the restricted product is locally compact by \cite[(1.1.13)]{NSW}.
\end{proof}

We recall the following lemma due to S.~Saito.

\begin{lemma}\label{s1.3}
Let $ \O$ be a henselian, discrete valuation ring with quotient field $K$, $m\ge 1$ an integer prime to the residue characteristic of\/ $\O$  and $\cF^\bullet$ a bounded, locally constant, constructible complex of \'{e}tale sheaves of\/ $\Z/m\Z$-modules on $\Spec \O$. Then the natural map
\[
H^n_\et(\Spec \O, \cF^\bullet) \to H^n_\et(\Spec K, \cF^\bullet)
\]
is injective for all $n$.
\end{lemma}

\begin{proof}
See \cite[Lemma 1.3]{Sa}.
\end{proof}
\begin{lemma} \label{slemma}
For almost all nonarchimedean places $v$, the restriction maps
\begin{equation}\label{injectcoh}
H^n_\et(\cX_{(v)},\cF^\bullet) \to H^n_\et(X_{(v)}, \cF^\bullet),
\end{equation}
\begin{equation}\label{injectcohc}
H^n_c(\cX_{(v)},\cF^\bullet) \to H^n_c(X_{(v)}, \cF^\bullet),
\end{equation}
\begin{equation}\label{injectext}
\Ext_{\cX_{(v)}}^{n-2}(\cF^\bullet,\Z^c_{\cX_{(v)}}(0))\to \Ext_{X_{(v)}}^n(\cF^\bullet,\Z^c_{X_{(v)}}(-1))
\end{equation}
are injective.
\end{lemma}
\begin{proof} By \cite[Thm.\,1.1]{fin}, there is a nonempty open subscheme $\cW\subset \cB$ such that $m$ is invertible on $\cW$ and the restriction of $Rf_*\cF^\bullet$ to $\cW$ is locally constant, constructible.  Let $\iota_v: \Spec \O_{(v)} \to \cB$ be the natural morphism. Then the assumptions of  \cref{s1.3} are satisfied for $\iota_v^*Rf_*\cF^\bullet$ showing that
\begin{equation}\label{inject66}
H^n_\et(\O_{(v)},Rf_*\cF^\bullet) \to H^n_\et(k_{(v)}, Rf_*\cF^\bullet)
\end{equation}
is injective for all $v\in \cW$.  Consider the fibre product diagram
\[
\begin{tikzcd}
\cX_{(v)}\rar{\iota_\cX}\arrow[d,"f_v"'] &\cX\dar{f}\\
\Spec \O_{(v)}\rar{\iota_v}&  \cB.
\end{tikzcd}
\]
Since \'{e}tale cohomology  commutes with inverse limits of quasi-compact, quasi-separated schemes and affine transition maps \cite[VII,\,5.8]{sga4}, we obtain an isomorphism
\begin{equation}
\iota_v^* Rf_*\cF^\bullet \cong Rf_{v *} \iota_\cX^* \cF^\bullet.
\end{equation}
Hence the injectivity of (\ref{inject66}) implies the injectivity of (\ref{injectcoh}) for $v\in \cW$.

The injectivity of (\ref{injectcohc}) follows by choosing a compactification $j: \cX_{(v)} \hookrightarrow \overline{\cX_{(v)}}$ and applying (\ref{injectcoh}) to $j_! \cF^\bullet$.

To show the injectivity of (\ref{injectext}), we first show that $Rf_*R\IHom_{\cX}(\cF^\bullet, \Z^c_{\cX}(0))$ is  bounded and constructible on a non-empty open subscheme  $\cW\subset \cB$.
Using \cref{dualgl}, \cref{wechselauftorsion} and  \cref{zc=gm}, we obtain
\begin{eqnarray*}
Rf_*R\IHom_{\cX}(\cF^\bullet,\Z^c_{\cX}(0))  &\cong&R\IHom_{\cB}(Rf_! \cF^\bullet,\Z^c_{\cB}(0)) \\
   &\cong& R\IHom_{\cB,\,\Z/m\Z}(Rf_! \cF^\bullet,\Z^c_{\cB}(0)/m[-1]) \\
   &\cong& R\IHom_{\cB,\,\Z/m\Z}(Rf_! \cF^\bullet,\mu_m).
\end{eqnarray*}
By \cite[XVII, Thm.\,5.3.6]{sga4}, $Rf_! \cF^\bullet$ is bounded and constructible on $\cB$. By \cite[Cor.\,1.6]{fin}, it follows that $R\IHom_{\cB,\,\Z/m\Z}(Rf_! \cF^\bullet,\mu_m)$ is bounded and constructible on $\cB[1/m]$.
We conclude that $Rf_*R\IHom_{\cX}(\cF^\bullet,\Z^c_{\cX}(0))$ is locally constant, constructible on a nonempty open subscheme $\cW\subset\cB[1/m]$. Applying \cref{s1.3} to $\iota_v^*Rf_*R\IHom_{\cX}(\cF^\bullet,\Z^c_{\cX}(0))$, we see that
\begin{multline}\label{inject77}
H^{n-2}_\et (\Spec \O_{(v)}, \iota_v^*Rf_*R\IHom_{\cX}(\cF^\bullet,\Z^c_{\cX}(0))) \to \\
  H^{n-2}_\et (\Spec k_{(v)}, \iota_v^*Rf_*R\IHom_{\cX}(\cF^\bullet,\Z^c_{\cX}(0)))
\end{multline}
is injective for all $v\in \cW$.
By \cref{extnummer2}, we obtain an isomorphism
\begin{equation}
\iota_v^*Rf_*R\IHom_{\cX}(\cF^\bullet,\Z^c_{\cX}(0)) \cong Rf_{v *} R\IHom_{\cX_{(v)}}(\iota_\cX^*\cF^\bullet,\Z^c_{\cX_{(v)}}(0)).
\end{equation}
Hence for $v\in \cW$, (\ref{inject77}) can be written as the injection
\begin{equation}
\Ext_{\cX_{(v)}}^{n-2}(\cF^\bullet,\Z^c_{\cX_{(v)}}(0))\hookrightarrow \Ext_{X_{(v)}}^{n-2}(\cF^\bullet,\Z^c_{\cX_{(v)}}(-1))=\Ext_{X_{(v)}}^n(\cF^\bullet,\Z^c_{X_{(v)}}(-1)).
\end{equation}
This finishes the proof.
\end{proof}

For a set of places $S$ and a finite subset $T\subset S$ we set
\begin{equation}
M^n_T(X,S,\cF^\bullet) = \prod_{v \in T} \widehat{\Ext}^n_{X_{(v)}}(\cF^\bullet, \Z^c_{X_{(v)}}(-1)) \times \prod_{v\in S\smallsetminus T} \Ext^{n-2}_{\cX_{(v)}}(\cF^\bullet, \Z^c_{\cX_{(v)}}(0)).
\end{equation}

\begin{corollary}\label{limesM} If we endow $M^n_T(X,S,\cF^\bullet)$ with the (compact) product topology, then there is a natural topological isomorphism
\begin{equation}
\ds Q^n_S(X, \cF^\bullet )\cong \varinjlim_{\substack{T\subset S\\T\, \mathrm{ finite}}}\ M^n_T(X,S,\cF^\bullet).
\end{equation}
\end{corollary}
\begin{proof}
This follows directly from  \cref{slemma} and the definition of the topology of the restricted product.
\end{proof}

\begin{definition}\label{defunramifiedcoh}
For non\-archi\-me\-dean~$v$ we define the unramified part
\begin{equation}
H^n_{\nr}(X_{(v)}, \cF^\bullet) \subset H^n_\et(X_{(v)}, \cF^\bullet)
\end{equation}
as the image of the restriction map
$H^n_\et(\cX_{(v)}, \cF^\bullet) \to H^n_\et(X_{(v)}, \cF^\bullet)$
and define
\begin{equation}
P^n_S(X,\cF^\bullet):= \resprod_{v\in S} \hat H^n_\et(X_{(v)},\cF^\bullet)
\end{equation}
as the restricted product with respect to the unramified subgroup (defined for nonarchimedean $v$). We define $P^n_{S,c}(X,\cF^\bullet)$ in exactly the same way but using modified cohomology with compact support everywhere.
\end{definition}

Again $P^n_S(X,\cF^\bullet)$ and $P^n_{S,c}(X,\cF^\bullet)$ only depend on $X$ and not of the choice of~$\cX$. They are locally compact abelian groups and  \cref{slemma} shows
\begin{corollary}\label{limesP} There are natural topological isomorphisms
\begin{equation}
\ds P^n_S(X, \cF^\bullet )\cong \varinjlim_{\substack{T\subset S\\T\, \mathrm{ finite}}} \quad \prod_{v\in T} H^n_\et(X_{(v)}, \cF^\bullet) \times  \prod_{v\in S\smallsetminus T} H^n_{\et}(\cX_{(v)}, \cF^\bullet),
\end{equation}
\begin{equation}
\ds P^n_{S,c}(X, \cF^\bullet )\cong \varinjlim_{\substack{T\subset S\\T\, \mathrm{ finite}}} \quad \prod_{v\in T} \hat H^n_c(X_{(v)}, \cF^\bullet) \times  \prod_{v\in S\smallsetminus T} H^n_{c}(\cX_{(v)}, \cF^\bullet) .
\end{equation}
\end{corollary}

Finally, we observe

\begin{proposition}\label{QdualzuPc}
The groups $P^n_{S,c}(X, \cF^\bullet )$ and $Q^{3-n}_S(X,\cF^\bullet)$ are Pontryagin dual to each other.
\end{proposition}

\begin{proof}
By \cref{Duality over local fields}, we have for all $v\in S$  a perfect pairing of finite groups
\begin{equation}\label{locdualpairing}
\hat H^n_c(X_{(v)}, \cF^\bullet) \times \widehat{\Ext}_{X_{(v)}}^{3-n}(\cF^\bullet,\Z^c_{X_{(v)}}(-1))\lang \Br(k_{(v)})\stackrel{\inv}\lang \Q/\Z.
\end{equation}
It therefore suffices to show that for almost all $v\in S$ the respective unramified subgroups are their exact annihilators in (\ref{locdualpairing}). First of all, they annihilate each other because the pairing (\ref{locdualpairing})  restricted to the subgroups  $H^n_c(\cX_{(v)}, \cF^\bullet)$ and $\Ext_{\cX_{(v)}}^{1-n}(\cF^\bullet,\Z^c_{\cX_{(v)}}(0))$ factors through $\Br(\O_{(v)})=0$.

On the other hand, $Rf_! \cF^\bullet$ is bounded and constructible on $\cB$ by \cite[XVII, Thm.\,5.3.6]{sga4}. Hence for almost all $v\in S$ the assumptions of  \cref{s1.3} are satisfied and the long exact sequence with support for $\Spec k_{(v)}\subset \Spec \O_{(v)}$  and $Rf_! \cF^\bullet$ splits  into short exact sequences
\begin{equation*}
0 \lang H^n_\et(\O_{(v)},Rf_! \cF^\bullet) \lang H^n_\et(k_{(v)},Rf_! \cF^\bullet)\lang H^{n+1}_v(\O_{(v)},Rf_!\cF^\bullet) \to 0.
\end{equation*}
Now we observe that
\[
\begin{array}{rcl}
H^n_c(\cX_{(v)}, \cF^\bullet)&=& H^n_\et(\O_{(v)},Rf_! \cF^\bullet),\\
H^n_c(X_{(v)}, \cF^\bullet)&=& H^n_\et(k_{(v)},Rf_! \cF^\bullet),
\end{array}
\]
and that $H^{n+1}_v(\O_{(v)},Rf_!\cF^\bullet)\cong \Ext_{\cX_{(v)}}^{1-n}(\cF^\bullet,\Z^c_{\cX_{(v)}}(0))^\vee$ by \cref{Local duality}. We obtain
\[
\# H^n_c(X_{(v)}, \cF^\bullet) = \# H^n_c(\cX_{(v)}, \cF^\bullet) \cdot  \# \Ext_{\cX_{(v)}}^{1-n}(\cF^\bullet,\Z^c_{\cX_{(v)}}(0)),
\]
and this equality of orders shows that the groups are their exact annihilators.
\end{proof}
\section{The main results}

\begin{theorem}\label{PText} Let $S\supset S_\infty$ be a (not necessarily finite) set of places of the global field $k$ and let $\cX_\cS \to \cS=\Spec \O_S$ be a separated scheme of finite type. Let $m\ge 1$ be an integer prime to $p=\ch k$ and let $\cF^\bullet$ be a bounded complex of constructible sheaves of\/ $\Z/m\Z$-modules on $(\cX_\cS)_\et$.

Then there is a natural long exact sequence of topological groups and strict homomorphisms
\[
\cdots \to \Ext^{n+1}_{\cX_\cS}(\cF^\bullet, \Z^c_{\cX_\cS}(0))\stackrel{\lambda_n}{\to} Q^{n+3}_S(X_k,\cF^\bullet)
  \to H^{-n}_c(\cX_\cS, \cF^\bullet)^\vee \to \cdots .
\]
The groups $\Ext^{n+1}_{\cX_\cS}(\cF^\bullet, \Z^c_{\cX_\cS}(0))$ are discrete, the groups $Q^{n+3}_S(X_k,\cF^\bullet)$ are locally compact, and the groups $H^{-n}_c(\cX_\cS, \cF^\bullet)^\vee$ are compact for all $n\in \Z$. Furthermore, the maps $\lambda_n$ are proper and have finite kernel for all $n\in \Z$.
\end{theorem}

\begin{proof}
We choose a separated scheme of finite type $f:\cX\to \cB=\Spec \O_\varnothing$ extending $\cX_\cS$, i.e., such that $\cX_\cS=\cX\times_\cB \cS$. Choosing $\cX$ small enough, we can assume that $\cF^\bullet$ is the restriction of a bounded constructible complex of sheaves of $\Z/m\Z$-modules on $\cX_\et$ to $(\cX_\cS)_\et$. We will denote this complex also by $\cF^\bullet$. We denote the common generic fibre of $\cX$ and $\cX_\cS$ by $X_k$.

In the following, the letter $T$ will always denote a finite subset $T\subset S$ containing all archimedean places and all places $v$ for which the assertion of \cref{slemma} fails. By $j: \cT \to \cB$ we denote the open immersion.

We apply
\cref{verglohnec} to the complex $Rf_!\cF^\bullet|_\cT=j^*Rf_!\cF^\bullet$ on $\cT_\et$ and obtain a long exact sequence
\begin{equation}\label{223}
\cdots \to \hat H^n_c(\cT,Rf_!\cF^\bullet|_\cT) \to H^n_\et(\cS,Rf_!\cF^\bullet|_\cS) \to L^n_T(S, Rf_! \cF^\bullet|_\cT)     \to \cdots  .
\end{equation}
We consider the terms in (\ref{223}). We have
\begin{equation}
\hat H^n_c(\cT,Rf_!\cF^\bullet|_\cT)\cong \hat H^n_c(\cX_\cT, \cF^\bullet) \cong \Ext^{2-n}_{\cX_\cT}(\cF^\bullet, \Z^c_{\cX_\cT}(0))^\vee
\end{equation}
by \cref{artinverdier}.
We have
\begin{eqnarray}
  L^n_T(S, j^*Rf_! \cF^\bullet) &=& \bigoplus_{v\in T}\hat H^n(k_{(v)}, Rf_!\cF^\bullet)  \oplus \bigoplus_{v\in S\smallsetminus T} H^{n+1}_v(\cT,Rf_!\cF^\bullet)\end{eqnarray}
\begin{multline}
 \cong \ \bigoplus_{v\in T}\widehat{\Ext}_{X_{(v)}}^{3-n}(\cF^\bullet,\Z^c_{X_{(v)}}(-1))^\vee\oplus \bigoplus_{v\in S\smallsetminus T} \Ext_{\cX_{(v)}}^{1-n}(\cF^\bullet,\Z^c_{\cX_{(v)}}(0))^\vee  \\
  = M^{3-n}_T(X_k, S, \cF^\bullet)^\vee
\end{multline}
by  \cref{Duality over local fields,Local duality}.
Finally, by \cite[XVII, 5.2.6]{sga4} we have
\begin{equation}
H^n_\et(\cS,Rf_!\cF^\bullet|_\cS) \cong H^n_c(\cX_\cS, \cF^\bullet),
\end{equation}
where the cohomology group on the right hand side is \'{e}tale cohomology with compact support of $\cX_\cS$ as a scheme of finite type over $\cS$. Hence we can write the dual sequence to (\ref{223}) in the form
\begin{equation}\label{mtseq}
\to  \Ext^{n+1}_{\cX_\cT}(\cF^\bullet, \Z^c_{\cX_\cT}(0))\to M^{n+3}_T(S,X_k,\cF^\bullet)
  \to H^{-n}_c(\cX_\cS, \cF^\bullet)^\vee \to \ldots
\end{equation}
(we changed the index $n\mapsto -n$ in order to have a cohomological complex). This is a long exact sequence of compact abelian groups with continuous maps and the Ext-groups are finite.
Passing to the direct limit over all finite $T\subset S$, we obtain using \cref{extnummer2} and \cref{limesM} the long exact sequence
\begin{equation} \label{qseq}
\cdots \to \Ext^{n+1}_{\cX_\cS}(\cF^\bullet, \Z^c_{\cX_\cS}(0))\stackrel{\lambda_n}{\to} Q^{n+3}_S(X_k,\cF^\bullet)
  \to H^{-n}_c(\cX_\cS, \cF^\bullet)^\vee \to \cdots .
\end{equation}
The groups $\Ext^{n+1}_{\cX_\cS}(\cF^\bullet, \Z^c_{\cX_\cS}(0))$ are discrete, the groups $Q^{n+3}_S(X_k,\cF^\bullet)$ are locally compact, and the groups $H^{-n}_c(\cX_\cS, \cF^\bullet)^\vee$ are compact.
Since all homomorphism in (\ref{mtseq}) are continuous, the same is true for the homomorphisms in (\ref{qseq}).

\medskip
Next we prove that all morphisms are strict. For this we have to show that  for any two consecutive maps $\phi$, $\psi$ in the long exact sequence, the continuous bijection $\im(\phi)\kiso \ker(\psi)$ is a  homeomorphism. Here $\im(\phi)$ is equipped with the quotient topology and $\ker(\psi)$ with the subspace topology. The obvious case is
\begin{equation*}
  H^{-n-1}_c(\cX_\cS, \cF^\bullet)^\vee \stackrel{\phi}{\to}\Ext^{n+1}_{\cX_\cS}(\cF^\bullet, \Z^c_\cX(0)) \stackrel{\psi}{\to} Q^{n+3}_S(X_k,\cF^\bullet).
\end{equation*}
Indeed, the group $\ker(\psi)$ is discrete, hence  $\im(\phi)\kiso \ker(\psi)$ must be homeomorphic. Since $H^{-n-1}_c(\cX_\cS, \cF^\bullet)^\vee$ is compact, we also obtain the assertion that $\ker(\psi)=\ker(\lambda_{n})$ is finite for all $n$. Next we consider
\begin{equation*}
 Q^{n+3}_S(X_k,\cF^\bullet)
  \stackrel{\phi}{\to} H^{-n}_c(\cX_\cS, \cF^\bullet)^\vee \stackrel{\psi}{\to} \Ext^{n+2}_{\cX_\cS}(\cF^\bullet, \Z^c_\cX(0)).
\end{equation*}
The image of $\phi$ is the union of the images of the maps
\begin{equation*}
\phi_T: M^{n+3}_T(X_k,S,\cF^\bullet) \to H^{-n}_c(\cX_\cS, \cF^\bullet)^\vee.
\end{equation*}
By (\ref{mtseq}) and the finiteness of $\Ext^{n+2}_{\cX_\cT}(\cF^\bullet, \Z^c_\cX(0))$, each $\im(\phi_T)$ has finite index in $H^{-n}_c(\cX_\cS, \cF^\bullet)^\vee$. Hence the images stabilize, i.e., $\im(\phi_T)=\im(\phi)$ for $T$ large enough. This shows that $\im(\phi)$ is compact and thus homeomorphic to $\ker(\psi)$. Finally, we consider
\begin{equation*}
\Ext^{n+1}_{\cX_\cS}(\cF^\bullet, \Z^c_\cX(0))\stackrel{\phi}{\to} Q^{n+3}_S(X_k,\cF^\bullet)
\stackrel{\psi}{\to}  H^{-n}_c(\cX_\cS, \cF^\bullet)^\vee.
\end{equation*}
Here we have to show that the subgroup $\ker(\psi)$ is discrete. For this it suffices to show that there is an open subgroup of $Q^{n+3}_S(X_k,\cF^\bullet)$ having finite intersection with $\ker(\psi)$. For sufficiently large finite $T\subset S$, $M_T(X_k,S,\cF^\bullet)$ is a subgroup of $Q^{n+3}_S(X_k,\cF^\bullet)$. Since $\Ext^{n+1}_{\cX_\cT}(\cF^\bullet, \Z^c_\cX(0))$ is finite,  any such $M_T(X_k,S,\cF^\bullet)$ has the required property.

Finally, the properness of $\lambda_n$ follows formally from what we already know:
Since $\ker(\lambda_n)$ is finite, it suffices to show that every compact subset of $Q^{n+3}_S(X_k,\cF^\bullet)$ has finite intersection with $\im(\lambda_n)$. But this is obvious since $\im(\lambda_n)$ is a discrete subgroup of $Q^{n+3}_S(X_k,\cF^\bullet)$ by the strictness of $\lambda_n$.
\end{proof}

Dualizing \cref{PText}, we obtain:

\begin{theorem}\label{PTextdual} Let $S\supset S_\infty$ be a (not necessarily finite) set of places of the global field $k$ and let $\cX_\cS \to \cS=\Spec \O_S$ be a separated scheme of finite type. Let $m\ge 1$ be an integer prime to $p=\ch (k)>0$ and let $\cF^\bullet$ be a bounded complex of constructible sheaves of\/ $\Z/m\Z$-modules on $(\cX_\cS)_\et$.

Then there is a natural long exact sequence of topological groups and strict homomorphisms
\[
\cdots \to H^{n}_c(\cX_\cS, \cF^\bullet)\stackrel{\lambda_{n,c}}{\to} P^{n}_{S,c}(X_k,\cF^\bullet)
  \to \Ext^{1-n}_{\cX_\cS}(\cF^\bullet, \Z^c_\cX(0))^\vee \to \cdots .
\]
The groups $H^{n}_c(\cX_\cS, \cF^\bullet)$ are discrete, the groups $P^{n}_{S,c}(X_k,\cF^\bullet)$ are locally compact, and the groups $\Ext^{1-n}_{\cX_\cS}(\cF^\bullet, \Z^c_\cX(0))^\vee$ are compact for all $n\in \Z$. Furthermore, the maps $\lambda_{n,c}$ are proper and have finite kernel for all $n\in \Z$.
\end{theorem}

\section{Proofs of Theorems A and B}\label{proofsect}

\begin{proof}[Proof of \cref{PTseq}] Now we assume that $\cX_\cS$ is regular and that $\cF$ is a locally constant, constructible sheaf of $\Z/m\Z$-modules where $m$ is invertible on $\cS$. We will deduce \cref{PTseq} from \cref{PText}. The scheme $\cX$ in \cref{PTseq} is $\cX_\cS$ of  \cref{PText} and the relative dimension $r$ occurring in \cref{PTseq}  is $d-1$, where $d=\dim \cX_\cS$ according to \cref{dimdef}. As before, we extend the situation to a scheme $\cX$ of finite type over $\cB=\Spec \O_\varnothing$.

\medskip\noindent
\underline{Step 1}. We have
\[
\Ext^{n+1}_{\cX_\cS}(\cF, \Z^c_\cX(0)) \cong H^{2d+n}_\et(\cX_\cS, \cF^\vee(d)),
\]
and for any nonarchimedean $v$
\[
\Ext^{n+1}_{\cX_{(v)}}(\cF, \Z^c_{\cX_{(v)}}(0)) \cong H^{2d+n}_\et(\cX_{(v)}, \cF^\vee(d)).
\]

\medskip\noindent
\underline{Proof of Step 1}. We have \renewcommand{\arraystretch}{2}
\[
\begin{array}{rcll}
\Ext^{n+1}_{\cX_\cS}(\cF, \Z^c_\cX(0))&\cong& \Ext^{n}_{\cX_\cS,\Z/m\Z}(\cF, \Z^c_\cX(0)/m)& \text{(\cref{wechselauftorsion})}\\
&\cong& \Ext^{2d+n}_{\cX_\cS,\Z/m\Z}(\cF, \mu_m^{\otimes d})&\text{(\cref{identifyZc})}\\
&\cong& H^{2d+n}_\et(\cX_\cS, \cF^\vee(d))&\text{(\cref{extdegen})}
\end{array}
\]
The proof of the second statement is similar.
\renewcommand{\arraystretch}{1}

\medskip\noindent
\underline{Step 2}. For $v\in S$ we have
\[
\Ext^{n+3}_{X_{(v)}}(\cF^\bullet, \Z^c_{X_{(v)}}(-1)) \cong H^{2d+n}_\et(X_{(v)}, \cF^\vee(d)).
\]

\medskip\noindent
\underline{Proof of Step 2}. We have
\renewcommand{\arraystretch}{1.8}
\[
\begin{array}{rcll}
\Ext^{n+3}_{X_{(v)}}(\cF, \Z^c_{X_{(v)}}(-1))&\cong& \Ext^{n+2}_{X_{(v)},\Z/m\Z}(\cF, \Z^c_{X_{(v)}}(-1)/m)& \text{(\cref{wechselauftorsion})}\\
&\cong& \Ext^{n+2d}_{X_{(v)},\Z/m\Z}(\cF, \mu_m^{\otimes d})&\text{(\cref{identifyZc0})}\\
&\cong& H^{2d+n}_\et(X_{(v)}, \cF^\vee(d))&\text{(\cref{extdegen})}.
\end{array}
\]
\renewcommand{\arraystretch}{1}

From Steps 1 and 2, we immediately obtain
\[
Q_S^{n+3}(X_k,\cF)= P_S^{2d+n}(X_k,\cF^\vee(d)).
\]
Applying  \cref{PText} to  $\cF^\vee (d)$, we obtain the exact sequence of \cref{PTseq} except at the boundaries.

\medskip\noindent
\underline{Step 3}.
\[
\lambda_0: H^0_\et(\cX_\cS,\cF) \to P_S^0(X_k,\cF)
\]
is injective and
\[
P_S^{2d}(X_k,\cF)\lang H^0_c(\cX_\cS,\cF^\vee(d))^\vee
\]
is surjective.

\medskip\noindent
\underline{Proof of Step 3}. The injectivity of $\lambda_0$ follows from our assumption that $S$ contains at least one nonarchimedean prime. The second map is dual to the injective map
\[
\lambda_{0,c}: H^0_c(\cX_\cS,\cF^\vee(d)) \to P_S^{2d}(X_k,\cF)^\vee\cong Q_S^3(X_k,\cF^\vee(d))^\vee \cong P_{S,c}^{0}(X_k,\cF^\vee(d)),
\] hence surjective.
This shows \cref{PTseq} (in the variant with henselizations instead of completions). For further use, we mention that for $i< 0$, our long exact sequence consists of isomorphisms
\begin{equation}\label{PTleft}
P^{i}_S(X_k,\cF)\liso H_c^{2r+2+i}(\cX_\cS,\cF^\vee(r+1))^\vee.
\end{equation}
\end{proof}

Dualizing, one obtains a version with compact supports:
\begin{thmABC}[Poitou-Tate exact sequence with compact support] \label{PTseqcs}  For $\cX$, $\cS$ and $\cF$ as in \cref{PTdual}, we have an exact $6r+9$-term sequence of abelian topological groups and strict homomorphisms
\begin{equation}\label{longexactPTcs}
\begin{tikzcd}[row sep=tiny, column sep=tiny]
0\rar&H^0_c(\cX, \cF)\rar{} & P^0_c(\cX,\cF)\rar & H^{2r+2}_\et(\cX,\cF^\vee(r+1))^\vee\rar&\null\\
&&\cdots &\\
\cdots\rar&H^i_c(\cX, \cF)\rar{\lambda_{i,c}} & P^i_c(\cX,\cF) \rar& H^{2r+2-i}_\et(\cX,\cF^\vee(r+1))^\vee\rar&\null\\
&&\cdots &\\
\cdots\rar&H^{2r+2}_c(\cX, \cF)\rar{} & P^{2r+2}_c(\cX,\cF)\rar & H^{0}_\et(\cX,\cF^\vee(r+1))^\vee\rar&0.
\end{tikzcd}
\end{equation}
Here
\begin{equation}
P^i_c(\cX,\cF):= \resprod_{v\in S} \hat H^i_c(\cX \otimes_{\O_S} k_v,\cF)
\end{equation}
is the restricted product with respect to the subgroups $H^i_{c,\nr}(\cX \otimes_{\O_S} k_v,\cF)$. The localization maps $\lambda_i$ are proper and have finite kernel for all $i$, and for $i\ge 2r+3$,
\begin{equation}
\lambda_{i,c}: H^i_c (\cX, \cF) \liso P^i_c(\cX,\cF)=\prod_{v\in S_\infty} \hat H^i_c(\cX \otimes_{\O_S} k_v,\cF)
\end{equation}
is an isomorphism. The groups in the left column of (\ref{longexactPTcs}) are discrete, those in the middle column locally compact, and those in the right column compact.
\end{thmABC}

\begin{proof}[Proof of Theorems \ref{PTdual} and \ref{PTseqcs}]
Applying Pontryagin duality to \cref{PTseq}, we obtain the sequence of \cref{PTseqcs} in view of
\[
P^{n}(\cX,\cF)^\vee\cong Q_S^{n+1-2r}(X_k,\cF^\vee(r+1))^\vee \cong P_{c}^{2r+2-n}(\cX,\cF^\vee(r+1))
\]
by \cref{QdualzuPc}. It remains to show the statement about $\lambda_{i,c}$. The isomorphism for $i\ge 2r+3$ follows by dualizing  from the last observation (\ref{PTleft}) in the proof of \cref{PTseq}.  Furthermore, the topological exactness of the sequence shows
\begin{equation}\label{shadual}
\ker (\lambda_{i,c}(\cF))\cong \ker(\lambda_{2r+3-i}(\cF^\vee(r+1)))^\vee.
\end{equation}
Hence the finiteness of $\ker(\lambda_i)$ for all $i$ shows the finiteness of $\ker(\lambda_{i,c})$ for all $i$.
Furthermore, (\ref{shadual}) shows \cref{PTdual} since
\[
\Sha^i(\cF)=\ker (\lambda_i(\cF)),\ \Sha^{2r+3-i}_c(\cF^\vee(r+1))=\ker(\lambda_{2r+3-i,c}(\cF^\vee(r+1))).\qedhere \]
\end{proof}
\section{Euler-Poincar\'{e} characteristic}

To complete the picture, we calculate the Euler-Poincar\'{e}-characteristic. We assume that the base field $k$ has no real embeddings so that $\cX$ has finite cohomological dimension. We also assume that $S$ is finite, hence the cohomology with values in constructible coefficients is finite. Let $\cF$ be a constructible sheaf on $\cX_\et$. Then we call
\begin{equation}\label{9.1}
\chi(\cX,\cF)= \prod_i \# H^i_\et(\cX,\cF)^{(-1)^i}
\end{equation}
the Euler-Poincar\'{e} characteristic of $\cF$. Let $k^s$ be a separable closure of $k$. Then we call
\begin{equation}
\chi^\mathrm{geo}(\cX,\cF):= \chi(X_{k^s},\cF|_{X_{k^s}})= \prod_i \# H^i_\et(X_{k^s},\cF|_{X_{k^s}})^{(-1)^i}
\end{equation}
the geometric Euler-Poincar\'{e} characteristic of $\cF$. We let $r_2$ be the number of complex places of $k$, hence by our assumptions
\begin{equation}
r_2=\left\{
\begin{array}{cc}
[k:\Q]/2,& \ch k=0,\\
0, & \ch k >0.
\end{array}\right.
\end{equation}
\begin{theorem} Under the above assumptions assume that $m\ge 1$ is an integer invertible on $\cS$ and $\cF$ a locally constant, constructible sheaf of\/ $\Z/m\Z$-modules on $\cX$. Then we have
\[
\chi(\cX,\cF) = \chi^\mathrm{geo}(\cX,\cF)^{-r_2}.
\]
In particular, $\chi(\cX,\cF)=1$ if  $\ch k>0$.
\end{theorem}

\begin{proof}
For a sheaf $\cG$ of $\Z/m\Z$-modules on $\cS$, we have by \cite[II, Thm.\,2.13]{ADT}
\begin{equation}\label{9.4}
\chi(\cS,\cG)= (\# \cG(k^s))^{-r_2}.
\end{equation}
For a bounded, constructible complex $\cG^\bullet$ of $\Z/m\Z$-modules on $\cS$, we put
\begin{equation}\label{9.5}
\chi(\cS,\cG^\bullet)= \prod_i \# H^i_\et(\cS,\cG^\bullet)^{(-1)^i}.
\end{equation}
Counting orders in the hypercohomology spectral sequence $E_2^{st}=H^s_\et(\cS,H^t(\cG^\bullet))\allowbreak\Rightarrow H_\et^{s+t}(\cS,\cG^\bullet)$, we obtain
\begin{equation}\label{9.6}
\chi(\cS,\cG^\bullet)=\ds\prod_{s,t} \# H^{s+t}_\et(\cS,\cG^\bullet)^{(-1)^{s+t}}
=\ds \prod_{s,t} (\#E_\infty^{s,t})^{(-1)^{s+t}}.
\end{equation}
Any differential of the spectral sequence goes from a group with $s+t=i$ to a group with $s+t=i+1$.  We therefore can replace the $E_\infty$-terms in (\ref{9.6}) by the $E_2$-terms and obtain
\renewcommand{\arraystretch}{1.5}
\[
\begin{array}{rcl}
\chi(\cS,\cG^\bullet)&=& \ds \prod_{s,t} (\#E_2^{s,t})^{(-1)^{s+t}}=\ds \prod_t\big(\prod_s (\#E_2^{s,t})^{(-1)^{s}} \big)^{(-1)^t}\\
&\stackrel{(\ref{9.1})}{=}&\ds\prod_t \chi(\cS,H^t(\cG^\bullet))\stackrel{(\ref{9.4})}{=} \ds\prod_i (\# H^i(\cG^\bullet)(k^s))^{(-1)^{i+1}r_2}.
\end{array} \renewcommand{\arraystretch}{1}
\]
Applying this to $\cG^\bullet=Rf_* \cF$, we obtain the result in view of $H^i(Rf_*\cF)(k^s)=H^i(X_{k^s},\cF|_{X_{k^s}})$.
\end{proof}

\section{Henselization versus completion}\label{vollstsec}
The results of this section allow us to replace henselization by completion in Theorems \ref{PText} and \ref{PTextdual}. In particular, this shows Theorems \ref{PTdual}, \ref{PTseq} and \ref{PTseqcs} in the way they are formulated.

\begin{proposition} \label{extnummer0}
Let $K$ be a henselian local field,
$f: X\to K$ separated and of finite type and $\cF^\bullet$ a bounded, constructible complex of sheaves on $X_\et$.
If $\ch (K)=p>0$ assume that $\cF^\bullet$ is $p$-torsion free.

Let $\widehat K$ be the completion of $K$ and
$\pi_X$ the base change $\pi_X: X_{\widehat K}=X \times_K \widehat K \to X$.
Then the natural morphisms
\renewcommand{\arraystretch}{1.4}
\[
\begin{array}{rcl}
  R\Hom_{X}(\cF^\bullet, \Z^c_X(n)) &\lang&
  R\Hom_{X_{\widehat K}}(\pi_X^*\cF^\bullet, \Z^c_{X_{\hat K}}(n)),\quad n\le 0,\\
  R\Gamma(X,\cF^\bullet) &\lang& R\Gamma(X_{\widehat K},\pi_X^*\cF^\bullet)\\
  R\Gamma_c(X,\cF^\bullet) &\lang& R\Gamma_c(X_{\widehat K},\pi_X^*\cF^\bullet)
\end{array}
\]
are isomorphisms in the derived category of abelian groups.
\end{proposition}

\begin{proof}
We start by proving the statements on $\R\Gamma$ and $\Gamma_c$. By \cite[X, 2.2.1]{sga4}, $K\to \hat K$ induces an isomorphism on absolute Galois groups. Hence we may replace $K$ and $\hat K$ by their separable closures and then the statement on $\R\Gamma$ is a well known consequence of the smooth base change theorem, cf.\ \cite[VI,\,4.3]{Mi}. Similarly, the statement on $R\Gamma_c$ follows since $Rf_!$ commutes with base change by \cite[XVII, 5.2.6]{sga4}.

Next we prove the assertion on $R\Hom$. Using the hyperext spectral sequence, we can assume that $\cF^\bullet$ is a single constructible sheaf $\cF$. If $i:Z\to X$ is a closed embedding with open complement
$j:U\subset X$, then we see by
comparing the localization triangles associated with the short exact sequence
$j_!j^*\cF \to \cF \to i_*i^*\cF$,
\[
R\Hom_{Z}(i^* \cF,\Z_Z^c(n))\to
R\Hom_{X}(\cF,\Z_X^c(n))\to
R\Hom_{U}(j^*\cF,\Z_U^c(n))
\]
that the statements for two of $Z,U$ and $X$ imply it for the third.

By induction on $d=\dim X$, we may therefore assume that $X$ is regular and connected, and $\cF$ is locally constant, constructible. Let $m\ge 1$ be an integer invertible in $K$ such that  $\cF$ is a sheaf of $\Z/m\Z$-modules. Then we have \renewcommand{\arraystretch}{1.3}
\[
\begin{array}{ccl}
R\Hom_{X}(\cF,\Z_X^c(n))   & \cong & R\Hom_{X, \Z/m\Z}(\cF,\Z_X^c(n)/m)[-1] \\
   & \cong & R\Hom_{X,\Z/m\Z}(\cF,\mu_m^{\otimes (d-n)})[2d-1]\\
   &\cong & R\Gamma(X,\IHom_{X,\Z/m\Z}(\cF, \mu_m^{\otimes d-n}))[2d-1]
\end{array}\renewcommand{\arraystretch}{1}
\]
by \cref{wechselauftorsion}, \cref{identifyZc0}, and \cref{extdegen}. The same holds for $X_{\widehat K}$ and $\pi_X^*\cF$. Hence, by the first part of the proof, it suffices to show that
\begin{equation}
\pi_X^*\IHom_{X,\Z/m\Z}(\cF, \mu_m^{\otimes d-n})\lang \IHom_{X_{\widehat{K}},\Z/m\Z}(\pi_X^*\cF, \mu_m^{\otimes d-n})
\end{equation}
is an isomorphism, which is clear since $\cF$ is locally constant.
\end{proof}

\begin{proposition} \label{extnummer1}
Let $K$ be a non-archimedean henselian local field, $B=\Spec \O_K$ and
$f: X\to B$ separated and of finite type.
Let $\cF^\bullet$ be a bounded, constructible complex of sheaves on $X_\et$.
If $\ch (K)=p>0$ assume that $\cF^\bullet$ is $p$-torsion free.

Let $\widehat K$ be the completion of $K$,
$\widehat B=\Spec \O_{\widehat K}$, $\pi: \widehat B\to B$ the projection and
$\pi_X$ the base change $\pi_X: X_{\widehat B}=X \times_B \widehat B \to X$.
Then the natural morphisms
\renewcommand{\arraystretch}{1.4}
\[
\begin{array}{rcl}
  R\Hom_{X}(\cF^\bullet, \Z^c_X(n)) &\lang&
  R\Hom_{X_{\widehat B}}(\pi_X^*\cF^\bullet, \Z^c_{X_{\hat B}}(n)),\quad n\le 0,\\
  R\Gamma_c(X,\cF^\bullet) &\lang& R\Gamma_c(X_{\widehat B},\pi_X^*\cF^\bullet)
\end{array}
\]
are isomorphisms in the derived category of abelian groups.

Assume in addition  that $\cF^\bullet$ is $p'$-torsion free, where $p'$ is the residue characteristic of $\O_K$. Then also
\[
 R\Gamma(X,\cF^\bullet) \lang R\Gamma(X_{\widehat B},\pi_X^*\cF^\bullet)
\]
is an isomorphism in the derived category of abelian groups.
\end{proposition}

\begin{proof}
As in the proof of \cref{extnummer0}, the statement for cohomology with compact support follows since $Rf_!$ commutes with base change by \cite[XVII, 5.2.6]{sga4}. Since $\O_K$ is excellent by assumption, the ring homomorphism $\O_K \to \O_{\widehat{K}}$ is regular. Hence, by Popescu's theorem \cite[Tag 07GB]{stacks}, $\O_{\widehat{K}}$, is the limit of smooth $\O_K$-algebras. Therefore the statement on cohomology follows from the smooth base change theorem.

Finally, we consider $R\Hom$. Considering $U \subset X$ open and $Z=X\smallsetminus U$, we see as in the proof of \cref{extnummer0}  that the statements for two of $Z,U$ and $X$ imply it for the third. As the closed fibres of $X$ and $X_{\hat B}$ coincide, it suffices to prove the statement for the generic fibre. Since $\Z^c_{X /B}(n)|_{X_K} \cong \Z^c_{X/K} (n-1)[2]$, this follows from \cref{extnummer0}.
\end{proof}

\end{document}